\documentclass{svjour3}                     
\smartqed  
\usepackage{bbm}
\usepackage{mathrsfs}
\usepackage{amsfonts}
\usepackage{stmaryrd}
\usepackage{graphicx}
\usepackage{subfigure}

\usepackage{bm}
\usepackage{amsmath}
\usepackage{amssymb}

\usepackage{amstext}
\usepackage{amsbsy}
\usepackage{setspace}
\usepackage{algorithm}
\usepackage{algorithmic}
\usepackage[figuresright]{rotating}
\usepackage{epstopdf}
\usepackage{makecell}

\numberwithin{equation}{section}
\numberwithin{theorem}{section}
\numberwithin{lemma}{section}
\numberwithin{remark}{section}
\usepackage{mathptmx}
\renewcommand{\figurename}{Fig.}

\newtheorem{thm}{Theorem}

\newtheorem{lem}[thm]{Lemma}
\newtheorem{prop}[thm]{Proposition}
\newtheorem{cor}[thm]{Corollary}

\newtheorem{rem}{Remark}
\newtheorem{ass}{Assumption}

\usepackage{mathptmx}      
\usepackage[justification=centering]{caption}

\date{}
\begin{document}

\title{Numerical approximation for fractional diffusion equation forced by a tempered fractional Gaussian noise}


\author{Xing Liu and Weihua Deng
}


\institute{
Xing Liu \at
              School of Mathematics and Statistics, Gansu Key Laboratory of Applied Mathematics and Complex
Systems, Lanzhou University, Lanzhou 730000, People's Republic of China. \email{2718826413@qq.com}
\and
Weihua Deng\at School of Mathematics and Statistics, Gansu Key Laboratory of Applied Mathematics and Complex
Systems, Lanzhou University, Lanzhou 730000, People's Republic of China.
\email{dengwh@lzu.edu.cn}
}

\maketitle

\begin{abstract}	

This paper discusses the fractional diffusion equation forced by a tempered fractional Gaussian noise. The fractional diffusion equation governs the probability density function of the subordinated killed Brownian motion. The tempered fractional Gaussian noise plays the role of fluctuating external source with the property of localization. We first establish the regularity of the infinite dimensional stochastic integration of the tempered fractional Brownian motion and then build the regularity of the mild solution of the fractional stochastic diffusion equation. The spectral Galerkin method is used for space approximation; after that the system is transformed into an equivalent form having better regularity than the original one in time. Then we use the semi-implicit Euler scheme to discretize the time derivative. In terms of the temporal-spatial error splitting technique, we obtain the error estimates of the fully discrete scheme in the sense of mean-squared $L^2$-norm. Extensive numerical experiments confirm the theoretical estimates.
\\
\\


\end{abstract}

\section{Introduction}
With the development of science and technology, nowadays anomalous diffusion phenomena are widely observed in our daily life. From statistics to mathematics, these phenomena have been modeled microscopically by stochastic processes, describing the paths of individual particles, and macroscopically by partial differential equations (PDEs), governing the probability density function (PDF) of a cloud of spreading particles \cite{Deng2020}. One of the effective ways to build the microscopic models of some anomalous diffusion phenomena is to subordinate the Brownian motion.

In recent years, two stochastic processes are very popular, which can be stated as follows. Let $D$ be a bounded region, $B(t)$ be a Brownian motion with $B(0)\in D$, and $\tau_D=\inf\{t>0: B(t)\notin D\}$. Denote $T_t$ as an $\alpha$-stable subordinator. The first stochastic process \cite{1} is given as
\begin{equation*}
 X_1(t)=\left\{
\begin{array}{cc}
B(T_t),\quad & t<\tau_D,\\
\Theta,\quad & t\ge\tau_D,
\end{array}
 \right.
 \end{equation*}
where $\Theta$ is a coffin state, meaning that the subordinated Brownian motion will be killed when first leaving the domain $D$; while the second process \cite{21-1} is
\begin{equation*}
 X_2(t)=\left\{
\begin{array}{cc}
B(T_t),\quad &T_t<\tau_D,\\
\Theta,\quad &T_t\ge\tau_D
\end{array}
 \right.
 \end{equation*}
with $\Theta$ still being a coffin state, implying to subordinate a killed Brownian motion (when first leaving the domain $D$).
The infinitesimal generator of $X_1(t)$ has the form
\begin{equation*}
    (-\Delta_1)^{\alpha}u(x)=c_{n,\alpha}{\rm P.V.}\int_{\mathbb{R}^n}\frac{u(x)-u(y)}{|x-y|^{n+2\alpha}}dy,\quad \alpha\in (0,1),
\end{equation*}
where  $c_{n,\alpha}=\frac{2^{2\alpha}\alpha\Gamma(n/2+\alpha)}{\pi^{n/2}\Gamma(1-\alpha)}$, ${\rm P.V.}$ denotes the principal value integral, and $u(y)=0$ for $y \in \mathbb{R}^n \backslash D$. We denote the infinitesimal generator of $X_2(t)$ as $(-\Delta)^{\alpha}$, and let $-\Delta$ be the infinitesimal generator of killed Brownian motion. It shows that  \cite{22,21-1}  if $\{(\lambda_i,\phi_i)\}^\infty_{i=1}$ are the eigenpairs of $-\mathrm{\Delta}$, then $\{(\lambda^\alpha_i,\phi_i)\}^\infty_{i=1}$ are the eigenpairs of $(-\mathrm{\Delta})^\alpha$, i.e.,
 \begin{equation}\label{eq:1.2}
 \left\{
\begin{array}{cc}
 -\mathrm{\Delta}\phi_i=\lambda_i\phi_i, \quad &\mathrm{in} \ D,\\
  \quad \phi_i=0,\quad &\mathrm{on} \ \partial D,
 \end{array}
 \right.
 \end{equation}
and
\begin{equation}\label{eq:1.4}
 \left\{
\begin{array}{cc}
(-\mathrm{\Delta})^\alpha\phi_i =\lambda^\alpha_i\phi_i, \quad &\mathrm{in} \ D,\\
  \quad \phi_i=0,\quad &\mathrm{on} \  \partial D.
 \end{array}
 \right.
 \end{equation}
The process concerned in this paper is $X_2(t)$ with infinitesimal generator $(-\Delta)^{\alpha}$. The governing equation for the PDF of positions of the particles $X_2(t)$ is
\begin{equation} \label{PDF}
 \left\{
\begin{array}{c}
\displaystyle\frac{\partial u(x,t)}{\partial t}=-(-\mathrm{\Delta})^\alpha u(x,t), \quad x \in  D,\\
u(x,t)=0,\quad (x,t)\in \partial D\times[0,T],
 \end{array}
 \right.
\end{equation}
where $T$ is a given positive real number.

Since the discovery of anomalous diffusion, the research of the corresponding nonlocal operators has attracted increasing attention. For instance, the physically meaningful and mathematically well-posed boundary conditions are investigated for the fractional PDEs in a bounded domain \cite{1}.
All kinds of numerical schemes are designed to solve the PDEs with nonlocal operators \cite{6,4,5,3,23-2}. And there are also a lot of discussions on the applications of the anomalous diffusion equations \cite{10,9,10-1,7,8}.

The aforementioned nonlocal PDEs are perfect to model anomalous diffusion in a stable environment. Considering the fickle natural environment, sometimes the stochastic disturbance of the source can not be ignored. The `localization' is one of the typical properties of fluctuation, well modeled by the tempered fractional Brownian motion (tfBm) \cite{24,23}.  The model we discuss in this paper is Eq. (\ref{PDF}) with the fluctuation source term (tempered fractional Gaussian noise) and the deterministic source term $f(u(x,t))$ depending on the concentration of the particles, i.e., the fractional stochastic PDEs
 \begin{equation}\label{eq:1.1}
\frac{\partial u(x,t)}{\partial t}=-(-\mathrm{\Delta})^\alpha u(x,t)+f\left(u(x,t)\right)+\dot{B}_{H,\mu}(x,t)
\end{equation}
with the initial and boundary conditions given by
\begin{eqnarray*}
 u(x,0)=u_0(x),\quad x\in D,\\
 u(x,t)=0,\quad (x,t)\in \partial D\times[0,T],
  \end{eqnarray*}
and $\alpha \in (0,1)$. Here $\dot{B}_{H,\mu}(x,t)=\frac{\partial B_{H,\mu}(x,t)}{\partial t}$ is the tempered fractional Gaussian noise; ${B}_{H,\mu}(x,t)$ represents an infinite dimensional tfBm and $H$ is Hurst parameter.

Nowadays, there are a lot of theoretical and numerical studies for the stochastic PDEs driven by the space-time white noise.
For example, the abundant theoretical consequences \cite{11,12,13} of the initial and boundary value problems for the stochastic PDEs are presented. In \cite{15,17,16}, the finite element approximation of some linear stochastic PDEs is studied in space. The numerical schemes of the semilinear or nonlinear stochastic PDEs are discussed \cite{18,20,19}. The work \cite{21} investigates a discrete approximation of the linear stochastic PDE with a positive-type memory. Instead of discussing the classical PDEs with white noise, this paper focuses on the fractional stochastic PDE driven by a tempered fractional Gaussian noise.  By using the Dirichlet eigenpairs, we first provide a complete solution theory of the mild solution $u(x,t)$ of Eq.\ \eqref{eq:1.1}, including existence, uniqueness, regularity in space, and the H\"older continuity in time. As for treating the stochastic integration of tfBm, unlike Brownian motion, the It$\mathrm{\hat{o}}$ isometry can not be directly used, on the contrary, we need to start from the definition of the stochastic integration of tfBm. Because of the property of $(-\mathrm{\Delta})^\alpha$, in some sense, the regularity of the solution of \eqref{eq:1.1} is determined by the regularity of its source terms. Further considering the nonlocal property of $(-\mathrm{\Delta})^\alpha$, the spectral Galerkin method is used to do the space approximation. The time discretization of the model also needs to be carefully dealt with because of the non H\"older continuity of the mild solution in some cases.

This paper is organized as follows. In the next section, we introduce some notations and preliminaries, including the definition of operators, assumptions, properties of tfBm, and stochastic integration. In section 3, we present the regularity of the mild solution of Eq.\ \eqref{eq:1.1} in the sense of mean-squared $L^2$-norm. In section 4, a spectral Galerkin space  semidiscretization of Eq.\ \eqref{eq:1.1} is given and the convergence is proved. In section 5, we derive the full discretization of Eq.\ \eqref{eq:1.1} and the convergence order for the proposed fully discrete scheme. The numerical experiments are performed in section 6. We conclude the paper with some discussions in the last section.
\section{Notations and preliminaries} \label{sec:2}
In this section, we give some notations, operators, properties of tfBm, and assumptions, which are commonly used in the paper.

Denote by $L^2(D;\mathbb{R})$ the Banach space consisting of $2$-times integrable function. Let $U=L^2(D;\mathbb{R})$ be a real separable Hilbert space with inner product  $\langle\cdot,\cdot\rangle$ and norm $\|\cdot\|=\langle\cdot,\cdot\rangle^{\frac{1}{2}}$. We define the unbounded linear operator $A^\nu$ by $A^\nu u=\left(-\mathrm{\Delta}\right)^\nu u$ on the domain 
\begin{equation*}
\mathrm{dom}\left(A^{\nu}\right)=\left\{u\in W^{1,2}_0\cap W^{2,2}: A^\nu u\in U\right\},
\end{equation*}
where $\nu\ge0$, $W^{\nu,2}$ is Sobolev space, and $W^{1,2}_0=\left\{u\in W^{1,2}: u(x)=0,\ x\in \partial D \right\}$. Then equation \eqref{eq:1.4} implies that there exists a sequence of increasing positive real numbers $\left\{\lambda_{i}\right\}_{i\in\mathbb{N}}$ and an orthonormal basis $\left\{\phi_{i}(x)\right\}_{i\in\mathbb{N}}$ of $U$ such that
\begin{equation*}
A^{\frac{\nu}{2}}\phi_{i}(x)=\lambda^{\frac{\nu}{2}}_{i}\phi_{i}(x)
\end{equation*}
and
\begin{equation*}
A^{\frac{\nu}{2}}u=\sum^\infty_{i=1}\lambda^{\frac{\nu}{2}}_{i}\left\langle u,\phi_{i}(x)\right\rangle\phi_{i}(x).
\end{equation*}
Moreover, we define the Hilbert space $\dot{U}^\nu=\mathrm{dom}\left(A^{\frac{\nu}{2}}\right)$
equipped with the inner product $\left\langle\cdot, \cdot\right\rangle_{\nu}=\left\langle A^{\frac{\nu}{2}}\cdot,A^{\frac{\nu}{2}}\cdot\right\rangle$ and norm $\|\cdot\|_\nu=\left\langle\cdot, \cdot\right\rangle^{\frac{1}{2}}_{\nu}$. In particular, $\dot{U}^0=U$.

\begin{lem}\label{le:Sec2}
{\rm(\cite{32,33})} Let $\Omega$ denote a bounded domain in $\mathbb{R}^d$, $d\in\{1,2,3\}$, and $|\Omega|$ the volume of $\Omega$. Let $\lambda_{i}$ be the i-th eigenvalue of the Dirichlet homogeneous boundary problem for the fractional Laplacian operator $(-\mathrm{\Delta})^\alpha$ in $\Omega$. 
Then
\begin{equation*}
\lambda_{i}\ge \frac{4d\pi^2 }{d+2}i^{\frac{2}{d}}\cdot|\Omega|^{-\frac{2}{d}}\cdot B^{-\frac{2}{d}}_d,
\end{equation*}
where $i\in\mathbb{N}$, and $B_d$ is the volume of the unit $d$-dimensional ball.
\end{lem}

\begin{ass}\label{as:2.1}
The function $f:U\to U$ satisfies
\begin{equation*}
\|f(u)-f(v)\|\lesssim\|u-v\|,  \quad u,v\in U
\end{equation*}
and
\begin{equation*}
\|A^{\frac{\nu}{2}} f(u)\|\lesssim 1+\|A^{\frac{\nu}{2}} u\|, \quad u\in \dot{U}^\nu.
\end{equation*}
\end{ass}

Next, we introduce the infinite dimensional tempered fractional Gaussian noise. The tfBm, describing anomalous diffusion with exponentially tempered long range correlations, was introduced in \cite{24,23} with its definition
\begin{equation*}
\beta_{H,\mu}(t)=\int_{-\infty}^{+\infty}\left[\mathrm{e}^{-\mu(t-y)_+}(t-y)_+^{H-\frac{1}{2}}
-\mathrm{e}^{-\mu(-y)_+}(-y)_+^{H-\frac{1}{2}}\right]\mathrm{d}\beta(y),
\end{equation*}
where $H>0$, $H\ne\frac{1}{2}$, $\mu>0$, $\beta(t)$ is standard Brownian motion and
\begin{equation*}
(y)_+=\left\{
\begin{array}{ll}
y, \quad y>0,\\
0, \quad y\le0.
\end{array}
\right.
\end{equation*}
The expectation of tfBm is
\begin{equation*}
\mathrm{E}\left[\beta_{H,\mu}(t)\right]=0,
\end{equation*}
and its covariance function is given as
\begin{equation}\label{eq:2.1}
\mathrm{E}\left[\beta_{H,\mu}(t)\beta_{H,\mu}(s)\right]=\frac{1}{2}\left[C^2_t|t|^{2H}+C^2_s|s|^{2H}-C^2_{t-s}|t-s|^{2H}\right],
\end{equation}
where
\begin{eqnarray*}
C^2_t&=&\int_{-\infty}^{+\infty}\left[\mathrm{e}^{-\mu |t|(1-y)_+}(1-y)_+^{H-\frac{1}{2}}
-\mathrm{e}^{-\mu|t|(-y)_+}(-y)_+^{H-\frac{1}{2}}\right]^2\mathrm{d}y\\
&=&\frac{2\Gamma(2H)}{(2\mu |t|)^{2H}}-\frac{2\Gamma\left(H+\frac{1}{2}\right)K_H(\mu |t|)}{\sqrt\pi(2\mu |t|)^{H}}
\end{eqnarray*}
for $t\ne0$ and $C^2_0=0$, $\Gamma(y)$ is the gamma function, and $K_H(y)$ is the modified Bessel function of the second kind
\begin{equation*}
K_H(y)=\frac{1}{2}\int^\infty_0t^{H-1}\exp\left[-\frac{1}{2}y\left(t+\frac{1}{t}\right)\right]\mathrm{d}t.
\end{equation*}

\begin{ass}\label{as:2.2}
Let driven stochastic process $B_{H,\mu}(x,t)$ be a cylindrical tfBm with respect to the normal filtration $\{\mathcal{F}_t\}_{t\in[0,T]}$. The process $B_{H,\mu}(x,t)$ can be represented by the formal series
\begin{equation*}
\dot{B}_{H,\mu}(x,t)=\sum^\infty_{i=1}\sigma_{i}\dot{\beta}^{i}_{H,\mu}(t)\phi_{i}(x),
\end{equation*}
where $|\sigma_{i}|\le \lambda_i^{-\rho}(\rho\ge0,\,\lambda_i$ is given in Lemma \ref{le:Sec2}$)$,  $\{\beta^{i}_{H,\mu}(t)\}_{i\in\mathbb{N}}$ are mutually independent real-valued tempered fractional Brownian motions, and $\left\{\phi_{i}(x)\right\}_{i\in\mathbb{N}}$ is an orthonormal basis  of $U$.
\end{ass}

We define $L^2(D,\dot{U}^\nu)$ to be the separable Hilbert space of square integrable random variables with norm
\begin{equation*}
\left\|u\right\|_{L^2\left(D,\dot{U}^\gamma\right)}=\left(\mathrm{E}\left[\left\|u\right\|^2_{\nu}\right]\right)^{\frac{1}{2}}, \quad \nu\ge0.
\end{equation*}
The theory of the stochastic integration for tfBm was developed in \cite{25}.
If $H>\frac{1}{2}$, $\mu>0$ and $g(y)\in L^2(\mathbb{R};\mathbb{R})$, the stochastic integral is
\begin{equation}\label{eq:2.2}
\int_\mathbb{R}g(y)\mathrm{d}\beta_{H,\mu}(y)=\Gamma(H+\frac{1}{2})\int_\mathbb{R}\left[\mathbb{I}^{H-\frac{1}{2},
\mu}g(y)-\mu\mathbb{I}^{H+\frac{1}{2},\mu}g(y)\right]\mathrm{d}\beta(y),
\end{equation}
where
\begin{equation*}
\mathbb{I}^{H,\mu}g(y)=\frac{1}{\Gamma\left(H\right)}
\int^{+\infty}_yg(x)(x-y)^{H-1}\mathrm{e}^{-\mu(x-y)}\mathrm{d}x.
\end{equation*}
For $0<H<\frac{1}{2}$, $\mu>0$ and $g(y)\in W^{\frac{1}{2}-H,2}(\mathbb{R};\mathbb{R})$, the stochastic integral is defined as
\begin{equation}\label{eq:2.3}
\int_\mathbb{R}g(y)\mathrm{d}\beta_{H,\mu}(y)=\Gamma(H+\frac{1}{2})\int_\mathbb{R}\left[\mathbb{D}^{\frac{1}{2}-H,
\mu}g(y)-\mu\mathbb{I}^{H+\frac{1}{2},\mu}g(y)\right]\mathrm{d}\beta(y),
\end{equation}
where
\begin{equation*}
\mathbb{D}^{H,\mu}g(y)=\mu^Hg(y)+\frac{H}{\Gamma\left(1-H\right)}
\int^{+\infty}_y\frac{g(y)-g(x)}{(x-y)^{H+1}}\mathrm{e}^{-\mu(x-y)}\mathrm{d}x.
\end{equation*}

\section{Regularity of the solution } \label{sec:3}
In this section, we give the formal mild solution of Eq.\ \eqref{eq:1.1} and prove the existence and uniqueness of the mild solution (the presentation is for the two dimensional case). Moreover, the H\"older continuity of the mild solution is discussed. These results will be used for numerical analysis.

For the sake of brevity, we rewrite Eq.\ \eqref{eq:1.1} as
\begin{equation}\label{eq:3.1}
\left\{
\begin{array}{ll}
\mathrm{d} u(t)+A^\alpha  u(t)\mathrm{d}t=f\left(u(t)\right)\mathrm{d}t+\mathrm{d}B_{H,\mu}(t),& \quad \mathrm{in} \ D\times(0,T],\\
u(0)=u_0,& \quad  \mathrm{in}\ D,\\
u(t)=0,& \quad \mathrm{on}\ \partial D,
\end{array}
\right.
\end{equation}
where $u(t)=u(x,t)$ and $B_{H,\mu}(t)=B_{H,\mu}(x,t)$. There is a formal mild solution $u(t)$ for Eq.\ \eqref{eq:3.1}, that is
\begin{equation}\label{eq:3.1-1}
u(t)=S(t)u_0+\int^t_0S(t-s)f\left(u(s)\right)\mathrm{d}s+\int^t_0S(t-s)\mathrm{d}B_{H,\mu}(s),
\end{equation}
where $S(t)=\mathrm{e}^{-tA^\alpha}$.
Eq.\ \eqref{eq:3.1-1} shows a fact that besides the initial value $u_0$, the regularity of the mild solution for Eq.\ \eqref{eq:3.1} depends on the stochastic integration $\int^t_0S(t-s)\mathrm{d}B_{H,\mu}(s)$. Therefore, we need to obtain following estimates in the first place.
\begin{prop}\label{le:1}
Let $0\le \widetilde{t}<t$, $0<\epsilon<2$, $\gamma=2\rho-1+(2-\epsilon)\alpha\cdot\min\{H,1\}$, and $\rho>\frac{1}{2}-\frac{(2-\epsilon)\alpha}{2}\cdot\min\{H,1\}$.
Then
\begin{equation*}
\mathrm{E}\left[\left\|A^{\frac{\gamma}{2}}\int^t_{\widetilde{t}}S(t-s)\mathrm{d}B_{H,\mu}(s)\right\|^2\right]\lesssim \frac{\left(t-\widetilde{t}\right)^{2H-\frac{(4-\epsilon)\cdot\min\{H,1\}}{2}}}{\alpha\left(\epsilon\cdot\min\{H,1\}\right)^2}.
\end{equation*}

\end{prop}
\begin{proof}
As $H>\frac{1}{2}$, by using Eq.\ \eqref{eq:2.2} and triangle inequality, we have
\begin{eqnarray*}
&\mathrm{E}&\left[\left\|A^{\frac{\gamma}{2}}\int^t_{\widetilde{t}}S(t-s)\mathrm{d}B_{H,\mu}(s)\right\|^2\right]\\
&\lesssim&\mathrm{E}\left[\left\|\sum_{i}\int^t_{\widetilde{t}}\frac{\lambda_{i}^{\frac{\gamma}{2}}\sigma_{i}\phi_{i}(x)}{\Gamma(H-\frac{1}{2})}\int^t_s\mathrm{e}^{-\lambda^\alpha_{i}(t-u)}(u-s)^{H-\frac{3}{2}}
\mathrm{e}^{-\mu(u-s)}\mathrm{d}u\mathrm{d}\beta^{i}(s)\right\|^2\right]\\
&+&\mathrm{E}\left[\left\|\sum_{i}\int^t_{\widetilde{t}}\frac{\lambda_{i}^{\frac{\gamma}{2}}\sigma_{i}\mu\phi_{i}(x)}{\Gamma(H+\frac{1}{2})}\int^t_s\mathrm{e}^{-\lambda^\alpha_{i}(t-u)}(u-s)^{H-\frac{1}{2}}
\mathrm{e}^{-\mu(u-s)}\mathrm{d}u\mathrm{d}\beta^{i}(s)\right\|^2\right]\\
&=&J_1+J_2.
\end{eqnarray*}
As $\theta_1<1$, the integration $\int^t_s\left(t-u\right)^{-\theta_1}(u-s)^{H-\frac{3}{2}}
\mathrm{d}u$ is finite. Let $\theta=\frac{(4-\epsilon)\cdot\min\{H,1\}}{4}$. Then combining Lemma \ref{le:Sec2}, It$\mathrm{\hat{o}}$'s isometry, $\mathrm{e}^{-x}\lesssim x^{-\theta_1} (x,\theta_1\ge0)$, and the mutual independence of $\beta^{i}(s)$, we have
\begin{eqnarray}\label{eq:3.2}
J_1&=&\sum_{i}\int^t_{\widetilde{t}}\left[\frac{\lambda_{i}^{\frac{\gamma}{2}}\sigma_{i}}{\Gamma(H-\frac{1}{2})}\int^t_s\mathrm{e}^{-\lambda^\alpha_{i}(t-u)}(u-s)^{H-\frac{3}{2}}
\mathrm{e}^{-\mu(u-s)}\mathrm{d}u\right]^2\mathrm{d}s\nonumber\\
&\lesssim&\sum_{i}\int^t_{\widetilde{t}}\left[\lambda_{i}^{\frac{\gamma}{2}}\lambda_{i}^{-\rho}\int^t_s\mathrm{e}^{-\lambda^\alpha_{i}(t-u)}(u-s)^{H-\frac{3}{2}}
\mathrm{d}u\right]^2\mathrm{d}s\nonumber\\
&\lesssim&\sum_{i}\int^t_{\widetilde{t}}\lambda_{i}^{-1-\frac{\epsilon\alpha\cdot\min\{H,1\}}{2}}\left[\int^t_s(t-u)^{-\theta}(u-s)^{H-\frac{3}{2}}
\mathrm{d}u\right]^2\mathrm{d}s\nonumber\\
&\lesssim&\sum_{i}\int^t_{\widetilde{t}}i^{-1-\frac{\epsilon\alpha\cdot\min\{H,1\}}{2}}(t-s)^{2H-2\theta-1}\mathrm{d}s\nonumber\\
&\lesssim&\frac{\left(t-\widetilde{t}\right)^{2H-\frac{(4-\epsilon)\cdot\min\{H,1\}}{2}}}{\alpha\left(\epsilon\cdot\min\{H,1\}\right)^2}.
\end{eqnarray}
For $J_2$, by using same steps, we have
\begin{equation}\label{eq:3.3}
J_2\lesssim
\frac{\left(t-\widetilde{t}\right)^{2H+2-\frac{(4-\epsilon)\cdot\min\{H,1\}}{2}}}{\alpha\epsilon\cdot\min\{H,1\}}.
\end{equation}

As $0<H<\frac{1}{2}$, by using Eq.\ \eqref{eq:2.3}, we have
\begin{eqnarray*}
&\mathrm{E}&\left[\left\|A^{\frac{\gamma}{2}}\int^t_{\widetilde{t}}S(t-s)\mathrm{d}B_{H,\mu}(s)\right\|^2\right]\nonumber\\
&\lesssim&\mathrm{E}\left[\left\|\sum_{i}\int^t_{\widetilde{t}}\lambda_{i}^{\frac{\gamma}{2}}\sigma_{i}\phi_{i}(x)\mathrm{e}^{-\lambda^\alpha_{i}(t-s)}\mathrm{d}\beta^{i}(s)\right\|^2\right]\nonumber\\
&+&\mathrm{E}\left[\left\|\sum_{i}\int^t_{\widetilde{t}}\lambda_{i}^{\frac{\gamma}{2}}\sigma_{i}\phi_{i}(x)\int^t_s\frac{\mathrm{e}^{-\lambda^\alpha_{i}(t-s)}-\mathrm{e}^{-\lambda^\alpha_{i}(t-u)}}{(u-s)^{\frac{3}{2}-H}}
\mathrm{e}^{-\mu(u-s)}\mathrm{d}u\mathrm{d}\beta^{i}(s)\right\|^2\right]\nonumber\\
&+&\mathrm{E}\left[\left\|\sum_{i}\int^t_{\widetilde{t}}\lambda_{i}^{\frac{\gamma}{2}}\sigma_{i}\mu\phi_{i}(x)\int^t_s\mathrm{e}^{-\lambda^\alpha_{i}(t-u)}(u-s)^{H-\frac{1}{2}}
\mathrm{e}^{-\mu(u-s)}\mathrm{d}u\mathrm{d}\beta^{i}(s)\right\|^2\right]\nonumber\\
&=&\widetilde{J}_1+\widetilde{J}_2+\widetilde{J}_3.
\end{eqnarray*}
Similar to the derivation of Eq.\ \eqref{eq:3.2}, we have
\begin{eqnarray}\label{eq:3.4}
\widetilde{J}_1&\lesssim&\sum_{i}\lambda_i^{\gamma-2\rho-\alpha}\left(1-\mathrm{e}^{-\lambda_i^\alpha\left(t-\widetilde{t}\right)}\right)\nonumber\\
&\lesssim&\frac{1}{\alpha(1-2H)}.
\end{eqnarray}
For the term $\widetilde{J}_2$, by using the fact that $\mathrm{e}^{-x}-\mathrm{e}^{-y}\lesssim |x-y|^{\theta_1}$ with $x,y\ge0$ and $0\le\theta_1\le1$, we get
\begin{eqnarray}\label{eq:3.5}
\widetilde{J}_2&=&\mathrm{E}\left[\left\|\sum_{i}\int^t_{\widetilde{t}}\lambda_{i}^{\frac{\gamma}{2}}\sigma_{i}\phi_{i}(x)\int^t_s\frac{\left(\mathrm{e}^{-\lambda^\alpha_{i}(u-s)}-1\right)\mathrm{e}^{-\lambda^\alpha_{i}(t-u)}}{(u-s)^{\frac{3}{2}-H}}
\mathrm{e}^{-\mu(u-s)}\mathrm{d}u\mathrm{d}\beta^{i}(s)\right\|^2\right]\nonumber\\
&\lesssim&\sum_{i}\int^t_{\widetilde{t}}\lambda_{i}^{\gamma-2\rho}\left(\int^t_s \left(\lambda^\alpha_{i}(u-s)\right)^\delta\frac{\left(\lambda^\alpha_{i}(t-u)\right)^{-\eta}}{(u-s)^{\frac{3}{2}-H}}
\mathrm{d}u\right)^2\mathrm{d}s\nonumber\\
&\lesssim&\sum_{i}\lambda_{i}^{-1-\frac{\epsilon\alpha H}{2}}\frac{\left(t-\widetilde{t}\right)^{\frac{\epsilon H}{2}}}{\epsilon H}\nonumber\\
&\lesssim&\frac{\left(t-\widetilde{t}\right)^{\frac{\epsilon H}{2}}}{\alpha\left(\epsilon H\right)^2}.
\end{eqnarray}
In second inequality, we choose $\delta>\frac{1}{2}-H$ and $\eta<1$ such that the integration$\int^t_s (u-s)^{\delta+H-\frac{3}{2}}(t-u)^{-\eta}\mathrm{d}u$ is bounded and $\eta-\delta=\frac{(4-\epsilon)H}{4}$.
For the term $\widetilde{J}_3$, we have
\begin{equation}\label{eq:3.6}
\widetilde{J}_3\lesssim\frac{\left(t-\widetilde{t}\right)^{2+\frac{\epsilon H}{2}}}{\alpha\epsilon H}.
\end{equation}
Then combining Eqs.\ \eqref{eq:3.2}-\eqref{eq:3.6} lead to
\begin{equation*}
\mathrm{E}\left[\left\|A^{\frac{\gamma}{2}}\int^t_{\widetilde{t}}S(t-s)\mathrm{d}B_{H,\mu}(s)\right\|^2\right]\lesssim \frac{\left(t-\widetilde{t}\right)^{2H-\frac{(4-\epsilon)\cdot\min\{H,1\}}{2}}}{\alpha\left(\epsilon\cdot\min\{H,1\}\right)^2},
\end{equation*}
which completes the proof.
\end{proof}

\begin{rem}\label{re:1}
As $\rho\ge 0.5$, Proposition \ref{le:1} holds for any $\alpha$ and $H$. If $\alpha H>0.5$ and $\alpha>0.5$, one can choose $\rho=0$.
\end{rem}

The following theorem shows the regularity results of the mild solution of Eq.\ \eqref{eq:3.1} by using Proposition \ref{le:1} and Dirichlet eigenpairs $\left\{\left(\lambda_{i},\phi_{i}(x)\right)\right\}_{i\in\mathbb{N}}$.

\begin{thm}\label{th:1}
Suppose that Assumptions \ref{as:2.1}-\ref{as:2.2} are satisfied, $\left\|u(0)\right\|_{L^2\left(D,\dot{U}^\gamma\right)}<\infty$, $0<\epsilon<2$, $\gamma=2\rho-1+(2-\epsilon)\alpha\cdot\min\{H,1\}$, and $\rho>\frac{1}{2}-\frac{(2-\epsilon)}{2}\alpha\cdot\min\{H,1\}$. Then Eq.\ \eqref{eq:3.1} possesses a unique mild solution 
\begin{equation*}
\left\|u(t)\right\|_{L^2(D,\dot{U}^\gamma)}\lesssim \frac{\alpha^{-\frac{1}{2}}}{\epsilon\cdot\min\{H,1\}}+\left\|u_0\right\|_{L^2\left(D,\dot{U}^\gamma\right)}.
\end{equation*}
Moreover, we have

$\mathrm{(i)}$ For $\rho>\frac{1}{2}$,
\begin{equation*}
\left\|u(t)-u(s)\right\|_{L^2(D,U)}\lesssim (t-s)^{\min\{H,1\}}\left(\frac{1}{\min\{\gamma, \alpha H,2\rho-1\}}+\left\|u_0\right\|_{L^2\left(D,\dot{U}^\gamma\right)}\right);
\end{equation*}

$\mathrm{(ii)}$ For $0<\rho\le\frac{1}{2}$,

\begin{eqnarray*}
\left\|u(t)-u(s)\right\|_{L^2(D,U)}\lesssim(t-s)^{\frac{\gamma }{2\alpha}}\left(\frac{1}{\epsilon\cdot\min\{H,1\}}+\left\|u_0\right\|_{L^2\left(D,\dot{U}^\gamma\right)}\right).
\end{eqnarray*}

\end{thm}
\begin{proof}
By using Proposition \ref{le:1} and triangle inequality, the regularity property of the mild solution $u(t)$ can be established as
\begin{eqnarray*}
\mathrm{E}\left[\left\|A^{\frac{\gamma}{2}}u(t)\right\|^2\right]&\lesssim& \mathrm{E}\left[\left\|\mathrm{e}^{-A^\alpha t}A^{\frac{\gamma}{2}}u(0)\right\|^2\right]\\
&&+\mathrm{E}\left[\left\|\int^t_0
\mathrm{e}^{-A^\alpha(t-s)}A^{\frac{\gamma}{2}}f\left(u(s)\right)\mathrm{d}s\right\|^2\right]+\frac{1}{\alpha\left(\epsilon\cdot\min\{H,1\}\right)^2}\\
&\lesssim&\mathrm{E}\left[\left\|A^{\frac{\gamma}{2}}u(0)\right\|^2\right]+\int^t_0\mathrm{E}\left[\left\|A^{\frac{\gamma}{2}}u(s)\right\|^2\right]\mathrm{d}s+\frac{1}{\alpha\left(\epsilon\cdot\min\{H,1\}\right)^2}.
\end{eqnarray*}
Then the Gr\"onwall inequality leads to
\begin{equation}\label{eq:3.6-1}
\left\|u(t)\right\|_{L^2(D,\dot{U}^\gamma)}\lesssim \frac{\alpha^{-\frac{1}{2}}}{\epsilon\cdot\min\{H,1\}}+\left\|u_0\right\|_{L^2\left(D,\dot{U}^\gamma\right)}.
\end{equation}
Meanwhile, we prove the H\"older continuity of the mild solution $u(t)$. Equation \eqref{eq:3.1-1} implies that
\begin{eqnarray*}
&\mathrm{E}&\left[\left\|u(t)-u(s)\right\|^2\right]\\
&\lesssim&\mathrm{E}\left[\left\|\left(S(t)-S(s)\right)u(0)\right\|^2\right]+\mathrm{E}\left[\left\|\int^t_sS(t-r)\mathrm{d}B_{H,\mu}(r)\right\|^2\right]\\
&+&\mathrm{E}\left[\left\|\int^s_0\left(S(t-r)-S(s-r)\right)\mathrm{d}B_{H,\mu}(r)\right\|^2\right]+\mathrm{E}\left[\left\|\int^t_sS(t-r)f\left(u(r)\right)\mathrm{d}r\right\|^2\right]\\
&+&\mathrm{E}\left[\left\|\int^s_0\left(S(t-r)-S(s-r)\right)f\left(u(r)\right)\mathrm{d}r\right\|^2\right]\\
&=&I_1+I_2+I_3+I_4+I_5.
\end{eqnarray*}
Let $\theta_1=\min\left\{\frac{\gamma}{2\alpha},1\right\}$. Then we have
\begin{eqnarray}\label{eq:3.10}
I_1&=&\mathrm{E}\left[\left\|\sum_{i}\left(\mathrm{e}^{-\lambda^\alpha_{i}t}-\mathrm{e}^{-\lambda^\alpha_{i}s}\right)
\left\langle u(0),\phi_{i}(x)\right\rangle\phi_{i}(x)\right\|^2\right]\nonumber\\
&\lesssim&\mathrm{E}\left[\left\|\sum_{i}(t-s)^{\theta_1}
\lambda^{\frac{\gamma}{2}}_{i}\left\langle u(0),\phi_{i}(x)\right\rangle\phi_{i}(x)\right\|^2\right]\nonumber\\
&\lesssim&(t-s)^{\min\left\{\frac{\gamma}{\alpha},2\right\}}\mathrm{E}\left[\left\|A^{\frac{\gamma}{2}} u(0)\right\|^2\right].
\end{eqnarray}
In the second inequality, we have used the fact $\gamma\ge2\alpha$, when $\theta_1=1$.

For $I_4$, we need to estimate the upper bound of $\mathrm{E}\left[\left\|u(t)\right\|^2\right]$. Let $\theta_1=\min\{H,1\}-\frac{\min\left\{H,\frac{\gamma}{2\alpha},1\right\}}{2}$.  Then Proposition \ref{le:1} implies
\begin{eqnarray*}
&&\mathrm{E}\left[\left\|\int^t_{\widetilde{t}}S(t-r)\mathrm{d}B_{H,\mu}(r)\right\|^2\right]\nonumber\\
&&\lesssim\sum_{i}\int^t_{\widetilde{t}}\left[\frac{\sigma_{i}}{\Gamma(H-\frac{1}{2})}\int^t_r\mathrm{e}^{-\lambda^\alpha_{i}(t-u)}(u-r)^{H-\frac{3}{2}}
\mathrm{e}^{-\mu(u-r)}\mathrm{d}u\right]^2\mathrm{d}r\nonumber\\
&&\lesssim\sum_{i}\int^t_{\widetilde{t}}\left[\lambda_{i}^{-\rho}\int^t_r\mathrm{e}^{-\lambda^\alpha_{i}(t-u)}(u-r)^{H-\frac{3}{2}}
\mathrm{d}u\right]^2\mathrm{d}r\nonumber\\
&&\lesssim\sum_{i}\int^t_{\widetilde{t}}\lambda_{i}^{-2\rho-2\alpha\theta_1}\left[\int^t_r\left(t-u\right)^{-\theta_1}(u-r)^{H-\frac{3}{2}}
\mathrm{d}u\right]^2\mathrm{d}r\nonumber\\
&&\lesssim\sum_{i}\int^t_{\widetilde{t}}\lambda_{i}^{-\gamma-1+\alpha\cdot\min\left\{H,\frac{\gamma}{2\alpha},1\right\}}(t-r)^{2H-2\theta_1-1}\mathrm{d}r\nonumber\\
&&\lesssim\frac{1}{\gamma\cdot\min\left\{2\alpha H,\gamma\right\}}.
\end{eqnarray*}
Finally, we have
\begin{eqnarray*}
\mathrm{E}\left[\left\|u(t)\right\|^2\right]\lesssim\frac{1}{\gamma\cdot\min\left\{2\alpha H,\gamma\right\}}+\mathrm{E}\left[\left\|u(0)\right\|^2\right]. \end{eqnarray*}
Combining the H\"older inequality and Assumption \ref{as:2.1} leads to
\begin{eqnarray*}
I_4&=& \mathrm{E}\left[\left\|\sum_{i}\int^t_s\mathrm{e}^{-\lambda^\alpha_{i}(t-r)}\left\langle f(u(r)),\phi_{i}(x)\right\rangle\phi_{i}(x)\mathrm{d}r\right\|^2\right]\\
&\lesssim& \mathrm{E}\left[\left\|\int^t_s f(u(r))\mathrm{d}r\right\|^2\right]\\
&\lesssim& (t-s)\mathrm{E}\left[\int^t_s \left(\left\|u(r)\right\|+1\right)^2\mathrm{d}r\right]\\
&\lesssim& (t-s)^2\left(\frac{1}{\gamma\cdot\min\left\{2\alpha H,\gamma\right\}}+\mathrm{E}\left[\left\| u(0)\right\|^2\right]\right).
\end{eqnarray*}
For $I_5$, similar to $I_4$, there is
\begin{eqnarray}\label{eq:sec3-1}
\mathrm{E}\left[\left\|A^{\frac{\gamma}{4}}u(t)\right\|^2\right]\lesssim\frac{1}{\gamma\cdot\min\left\{2\alpha H,\gamma\right\}}+\mathrm{E}\left[\left\|A^{\frac{\gamma}{4}}u(0)\right\|^2\right].
\end{eqnarray}
Let $\theta_2=\min\left\{\frac{\gamma}{2\alpha},1\right\}$. Then
\begin{eqnarray*}
I_5&=&\mathrm{E}\left[\left\|\sum_{i}\int^s_0\left(\mathrm{e}^{-\lambda^\alpha_{i}(t-s)}-1\right)\mathrm{e}^{-\lambda^\alpha_{i}(s-r)}
\left\langle f(u(r)),\phi_{i}(x)\right\rangle\phi_{i}(x)\mathrm{d}r\right\|^2\right]\\
&\lesssim&(t-s)^{2\theta_2}\mathrm{E}\left[\left\|\sum_{i}\int^s_0\lambda^{\frac{\alpha\theta_2}{2}}_{i}(s-r)^{-\frac{\theta_2}{2}}
\left\langle f(u(r)),\phi_{i}(x)\right\rangle\phi_{i}(x)\mathrm{d}r\right\|^2\right]\\
&\lesssim&(t-s)^{2\theta_2}\mathrm{E}\left[\left\|\int^s_0(s-r)^{-\frac{\theta_2}{2}}
A^{\frac{\gamma}{4}}f(u(r))\mathrm{d}r\right\|^2\right]\\
&\lesssim&(t-s)^{2\theta_2}\mathrm{E}\left[\left(\int^s_0(s-r)^{-\frac{\theta_2}{2}}
\left(\left\|A^{\frac{\gamma}{4}}u(r)\right\|+1\right)\mathrm{d}r\right)^2\right]\\
&\lesssim&(t-s)^{2\theta_2}\mathrm{E}\left[\int^s_0(s-r)^{-\frac{\theta_2}{2}}
\left(\left\|A^{\frac{\gamma}{4}}u(r)\right\|+1\right)^2\mathrm{d}r\int^s_0(s-r)^{-\frac{\theta_1}{2}}\mathrm{d}r\right]\\
&\lesssim&(t-s)^{\min\left\{\frac{\gamma}{\alpha},2\right\}}\left(\frac{1}{\gamma\cdot\min\left\{2\alpha H,\gamma\right\}}+\mathrm{E}\left[\left\| A^{\frac{\gamma}{4}}u(0)\right\|^2\right]\right).
\end{eqnarray*}

As $\rho>\frac{1}{2}$, Proposition \ref{le:1} leads to
\begin{eqnarray}\label{eq:sec3-2}
I_2&=&\mathrm{E}\left[\left\|\int^t_sS(t-r)\mathrm{d}B_{H,\mu}(r)\right\|^2\right]\lesssim\frac{(t-s)^{2H}}{(2\rho-1)H}.
\end{eqnarray}
Let $\theta_3=\min\{H,1\}$, $\theta_4=\theta_3-\min\left\{\frac{2\rho-1}{4\alpha},\frac{\theta_3}{2}\right\}$, and $H>\frac{1}{2}$. Using Eq.\ \eqref{eq:2.2} and Proposition \ref{le:1} leads to
\begin{eqnarray*}
I_3&\lesssim& \mathrm{E}\left[\left\|\sum_{i}\int^s_0\frac{\sigma_{i}\phi_{i}(x)}{\Gamma(H-\frac{1}{2})}\int^s_r\left(\mathrm{e}^{-\lambda^\alpha_{i}(t-u)}-\mathrm{e}^{-\lambda^\alpha_{i}(s-u)}
\right)(u-r)^{H-\frac{3}{2}}
\mathrm{e}^{-\mu(u-r)}\mathrm{d}u\mathrm{d}\beta^{i}(r)\right\|^2\right]\\
&+&\mathrm{E}\left[\left\|\sum_{i}\int^s_0\frac{\sigma_{i}\mu\phi_{i}(x)}{\Gamma(H+\frac{1}{2})}\int^s_r\left(\mathrm{e}^{-\lambda^\alpha_{i}(t-u)}-\mathrm{e}^{-\lambda^\alpha_{i}(s-u)}
\right)(u-r)^{H-\frac{1}{2}}
\mathrm{e}^{-\mu(u-r)}\mathrm{d}u\mathrm{d}\beta^{i}(r)\right\|^2\right]\\
&\lesssim&(t-s)^{2\theta_3}\sum_{i}\int^s_0\lambda_{i}^{2\alpha\theta_3}\sigma^2_{i}\left(\int^s_r\mathrm{e}^{-\lambda^\alpha_{i}(s-u)}(u-r)^{H-\frac{3}{2}}\mathrm{d}u\right)^2\mathrm{d}r\\
&+&(t-s)^{2\theta_3}\sum_{i}\int^s_0\lambda_{i}^{2\alpha\theta_3}\sigma^2_{i}\left(\int^s_r\mathrm{e}^{-\lambda^\alpha_{i}(s-u)}(u-r)^{H-\frac{1}{2}}\mathrm{d}u\right)^2\mathrm{d}r\\
&\lesssim&(t-s)^{\min\{2H,2\}}\sum_{i}i^{2\alpha\left(\theta_3-\theta_4\right)-2\rho}\int^s_0\left((s-r)^{2H-2\theta_4-1}+(s-r)^{2H-2\theta_4+1}\right)\mathrm{d}r\\
&\lesssim&(t-s)^{\min\{2H,2\}}\sum_{i}i^{2\alpha\left(\theta_3-\theta_4\right)-2\rho}\frac{1}{H-\theta_4}\\
&\lesssim&\frac{(t-s)^{\min\{2H,2\}}}{(2\rho-1) H}.
\end{eqnarray*}
As $0<H<\frac{1}{2}$, by using the same procedure, we have $I_2+I_3\lesssim \frac{(t-s)^{\min\{2H,2\}}}{(2\rho-1) H}$.
Then
\begin{equation*}
\left\|u(t)-u(s)\right\|_{L^2(D,U)}\lesssim (t-s)^{\min\{H,1\}}\left(\frac{1}{\min\{\gamma, \alpha H,2\rho-1\}}+\left\|u_0\right\|_{L^2\left(D,\dot{U}^\gamma\right)}\right).
\end{equation*}
For $0<\rho\le\frac{1}{2}$, as $H>\frac{1}{2}$, Proposition \ref{le:1} leads to
\begin{eqnarray}\label{eq:3.11}
I_2&=&\mathrm{E}\left[\left\|\int^t_sS(t-r)\mathrm{d}B_{H,\mu}(r)\right\|^2\right]\nonumber\\
&\lesssim&\mathrm{E}\left[\left\|\sum_{i}\int^t_{s}\frac{\sigma_{i}\phi_{i}(x)}{\Gamma(H-\frac{1}{2})}\int^t_r\mathrm{e}^{-\lambda^\alpha_{i}(t-u)}(u-r)^{H-\frac{3}{2}}
\mathrm{e}^{-\mu(u-r)}\mathrm{d}u\mathrm{d}\beta^{i}(r)\right\|^2\right]\nonumber\\
&\lesssim&\sum_{i}\lambda^{-2\rho-\left(1 -2\rho+\epsilon\alpha\cdot\min\{H,1\}\right)}_{i}\int^t_{s}\left(\int^t_r(t-u)^{-\frac{1-2\rho+\epsilon\alpha\cdot\min\{H,1\}}{2\alpha}}(u-r)^{H-\frac{3}{2}}\mathrm{d}u\right)^2\mathrm{d}r\nonumber\\
&\lesssim&\sum_{i}i^{-1-\epsilon\alpha\cdot\min\{H,1\}}\int^t_{s}(t-r)^{-\frac{1-2\rho+\epsilon\alpha\cdot\min\{H,1\}}{\alpha}+2H-1}\mathrm{d}r\nonumber\\
&\lesssim&\frac{1}{\epsilon\alpha H\gamma}(t-s)^{\frac{2\alpha H-1+2\rho-\epsilon\alpha\cdot\min\{H,1\}}{\alpha}}.
\end{eqnarray}
In the second inequality, we have used the fact that $2\alpha>1-2\rho+\epsilon\alpha\cdot\min\{H,1\}$. The condition $\rho\le\frac{1}{2}$ leads to $\gamma<2\alpha$. Thus we have
\begin{eqnarray}\label{eq:3.11-1}
I_3&\lesssim& \mathrm{E}\left[\left\|\sum_{i}\int^s_0\frac{\sigma_{i}\phi_{i}(x)}{\Gamma(H-\frac{1}{2})}\int^s_r\left(\mathrm{e}^{-\lambda^\alpha_{i}(t-u)}-\mathrm{e}^{-\lambda^\alpha_{i}(s-u)}
\right)(u-r)^{H-\frac{3}{2}}
\mathrm{e}^{-\mu(u-r)}\mathrm{d}u\mathrm{d}\beta^{i}(r)\right\|^2\right]\nonumber\\
&+&\mathrm{E}\left[\left\|\sum_{i}\int^s_0\frac{\sigma_{i}\mu\phi_{i}(x)}{\Gamma(H+\frac{1}{2})}\int^s_r\left(\mathrm{e}^{-\lambda^\alpha_{i}(t-u)}-\mathrm{e}^{-\lambda^\alpha_{i}(s-u)}
\right)(u-r)^{H-\frac{1}{2}}
\mathrm{e}^{-\mu(u-r)}\mathrm{d}u\mathrm{d}\beta^{i}(r)\right\|^2\right]\nonumber\\
&\lesssim&(t-s)^{\frac{\gamma }{\alpha}}\sum_{i}\int^s_0\lambda_{i}^{\gamma-2\rho- \frac{(4-\epsilon)\alpha\cdot\min\{H,1\}}{2}}\left(\int^s_r(s-u)^{-\frac{(4-\epsilon)\cdot\min\{H,1\}}{4}}(u-r)^{H-\frac{3}{2}}\mathrm{d}u\right)^2\mathrm{d}r\nonumber\\
&\lesssim&(t-s)^{\frac{\gamma }{\alpha}}\sum_{i}\lambda_{i}^{-1-\frac{\epsilon\alpha\cdot\min\{H,1\}}{2}}\frac{1}{\epsilon\cdot\min\{H,1\}}\nonumber\\
&\lesssim&\frac{(t-s)^{\frac{\gamma }{\alpha}}}{(\epsilon\cdot\min\{H,1\})^2}.
\end{eqnarray}
Similarly, for $0<H<\frac{1}{2}$, we have $I_2+I_3\lesssim\frac{(t-s)^{\frac{\gamma }{\alpha}}}{(\epsilon\cdot\min\{H,1\})^2}$. Due to $2\alpha H-1+2\rho-\epsilon\alpha\cdot\min\{H,1\}\le\gamma$, then
\begin{eqnarray*}
\left\|u(t)-u(s)\right\|_{L^2(D,U)}\lesssim(t-s)^{\frac{\gamma }{2\alpha}}\left(\frac{1}{\epsilon\cdot\min\{H,1\}}+\left\|u_0\right\|_{L^2\left(D,\dot{U}^\gamma\right)}\right),
\end{eqnarray*}
which completes the proof.
\end{proof}

See Appendix for the existence and uniqueness of the mild solution of Eq. \ \eqref{eq:3.1-1}.

\section{Galerkin approximation for spatial discretization} \label{sec:4}
In this section, we provide the Galerkin spatial semi-discretization of Eq.\ \eqref{eq:3.1}. The error estimates are also presented.

To implement the Galerkin spatial approximation of Eq.\ \eqref{eq:3.1}, we choose a finite dimensional subspace of $U$. Let $U^N$ be a $N$ dimensional subspace of $U$, and the sequence $\left\{\phi_{1}(x), \dots, \phi_{i}(x), \dots, \phi_{N}(x)\right\}_{N\in\mathbb{N}}$ is an orthonormal basis of $U^N$. Then we introduce the projection operator $P_N:\,U\to U^N$: for $\xi\in U$,
\begin{equation*}
P_N\xi=\sum^N_{i=1}\left\langle\xi,\phi_{i}(x) \right\rangle\phi_{i}(x)
\end{equation*}
and
\begin{equation}\label{eq:4.1}
\left\langle P_N\xi,\chi\right\rangle=\left\langle \xi,\chi\right\rangle \quad \forall \chi\in U^N.
\end{equation}
Additionally, define $A^\alpha_N:\, U\to U^N$ with 
\begin{equation}\label{eq:4.2}
\left\langle A^\alpha_N\xi,\chi\right\rangle=\left\langle A^\alpha\xi,\chi\right\rangle\quad \forall  \chi\in U^N,
\end{equation}
and it has $A^\alpha_N=A^\alpha P_N$.

The Galerkin formulation of Eq.\ \eqref{eq:3.1} is: Find $u^N(t)\in U^N$ such that
\begin{equation}\label{eq:4.3}
\left\{
\begin{array}{ll}
\left\langle\mathrm{d}u^N(t),\chi\right\rangle+\left\langle A^\alpha u^N(t)\mathrm{d}t,\chi\right\rangle=
\left\langle f\left(u^N(t)\right)\mathrm{d}t,\chi\right\rangle+\left\langle\mathrm{d}B_{H,\mu}(t),\chi\right\rangle,\quad\chi\in U^N,\\
\left\langle u^N(0),\chi\right\rangle=\left\langle u(0),\chi\right\rangle.
\end{array}
\right.
\end{equation}
Then according to Eq.\ \eqref{eq:4.1}, Eq.\ \eqref{eq:4.2}, and Eq.\ \eqref{eq:4.3}, the Galerkin approximation of Eq.\ \eqref{eq:3.1} is obtained
\begin{equation}\label{eq:4.4}
\left\{
\begin{array}{ll}
\mathrm{d} u^N(t)+A^\alpha_Nu^N(t)\mathrm{d}t=f_N\left(u^N(t)\right)\mathrm{d}t+P_N\mathrm{d}B_{H,\mu}(t),\quad t\in(0,T],\\
u^N(0)=P_Nu(0),
\end{array}
\right.
\end{equation}
where $f_N=P_Nf$. Similar to the Eq.\ \eqref{eq:3.1}, the unique mild solution of Eq.\ \eqref{eq:4.4} is given by
\begin{equation}\label{eq:4.4-4}
u^N(t)=S_N(t)u_0+\int^t_0S_N(t-s)f\left(u^N(s)\right)\mathrm{d}s+\int^t_0S_N(t-s)\mathrm{d}B_{H,\mu}(s),
\end{equation}
where $S_N(t)=\mathrm{e}^{-tA^\alpha_N}$. Theorem \ref{th:1} implies the following result.
\begin{cor} \label{cor:1}
Suppose that Assumptions \ref{as:2.1}-\ref{as:2.2} are satisfied, $\left\|u(0)\right\|_{L^2(D,\dot{U}^\gamma)}<\infty$, $0<\epsilon<2$, $\gamma=2\rho-1+(2-\epsilon)\alpha\cdot\min\{H,1\}$, $\rho>\frac{1}{2}-\frac{(2-\epsilon)}{2}\alpha\cdot\min\{H,1\}$,
and $u^N(t)$ is the unique mild solution of Eq.\ \eqref{eq:4.4}.
Then
\begin{equation*}
\left\|u^N(t)\right\|_{L^2\left(D,\dot{U}^\gamma\right)}\lesssim \frac{\alpha^{-\frac{1}{2}}}{\epsilon\cdot\min\{H,1\}}+\left\|u_0\right\|_{L^2\left(D,\dot{U}^\gamma\right)}
\end{equation*}
and we obtain the H\"older regularity of the mild solution $u^N(t)$:

$\mathrm{(i)}$ For $\rho>\frac{1}{2}$,

\begin{equation*}
\left\|u^N(t)-u^N(s)\right\|_{L^2(D,U)}\lesssim
(t-s)^{\min\{H,1\}}\left(\frac{1}{\min\{\gamma, \alpha H,2\rho-1\}}+\left\|u_0\right\|_{L^2\left(D,\dot{U}^\gamma\right)}\right);
\end{equation*}

$\mathrm{(ii)}$ For $0<\rho\le\frac{1}{2}$,

\begin{eqnarray*}
\left\|u^N(t)-u^N(s)\right\|_{L^2\left(D,\dot{U}^\gamma\right)}\lesssim(t-s)^{\frac{\gamma }{2\alpha}}\left(\frac{1}{\epsilon\cdot\min\{H,1\}}+\left\|u_0\right\|_{L^2\left(D,\dot{U}^\gamma\right)}\right).
\end{eqnarray*}

\end{cor}

To analyze the error of the Galerkin spatial semi-discretization for Eq.\ \eqref{eq:3.1}, the following lemmas are needed.
\begin{lem}\label{le:2}
If $\mathrm{E}\left[\|A^{\frac{\nu}{2}}\xi\|^2\right]<\infty$, $\xi\in U$, then
\begin{equation*}
\mathrm{E}\left[\|(P_N-I)\xi\|^2\right]\lesssim \lambda_{N+1}^{-\nu}\mathrm{E}\left[\|A^{\frac{\nu}{2}}\xi\|^2\right].
\end{equation*}
\end{lem}
\begin{proof}
\begin{eqnarray*}
\mathrm{E}\left[\|(P_N-I)\xi\|^2\right]&=&\mathrm{E}\left[\left\|\sum^\infty_{i=N+1}\left\langle\xi,\phi_{i}(x)\right\rangle\phi_{i}(x)\right\|^2\right]\\
&\lesssim& \lambda_{N+1}^{-\nu}\mathrm{E}\left[\left\|\sum^\infty_{i=N+1}\lambda^{\frac{\nu}{2}}_{i}\left\langle\xi,\phi_{i}(x)\right\rangle\phi_{i}(x)\right\|^2\right]\\
&\lesssim& \lambda_{N+1}^{-\nu}\mathrm{E}\left[\|A^{\frac{\nu}{2}}\xi\|^2\right].
\end{eqnarray*}
\end{proof}

From Lemma \ref{le:2}, one can infer the bound of the spatial error for scheme \eqref{eq:4.4} in the $L^2(D,U)$.
\begin{thm} \label{th:2}
Let $u(t)$ and $u^N(t)$ be, respectively, the mild solutions of Eq.\ \eqref{eq:3.1} and Eq.\ \eqref{eq:4.4} with the assumptions given in Theorem \ref{th:1} and Corollary \ref{cor:1}. 
Then we have
\begin{equation*}
\left\|u(t)-u^N(t)\right\|_{L^2(D,U)}\lesssim \lambda_{N+1}^{-\frac{\gamma}{2}}\left(\frac{\alpha^{-\frac{1}{2}}}{\epsilon\cdot\min\{H,1\}}+\left\|u_0\right\|_{L^2\left(D,\dot{U}^\gamma\right)}\right).
\end{equation*}
\end{thm}
\begin{proof}
Using the triangle inequality, Lemma \ref{le:2}, and Theorem \ref{th:1}, we obtain
\begin{eqnarray*}
&\mathrm{E}&\left[\|u(t)-u^N(t)\|^2\right]\\
&\lesssim&\mathrm{E}\left[\|u(t)-P_Nu(t)\|^2\right]+\mathrm{E}\left[\|P_Nu(t)-u^N(t)\|^2\right]\\
&\lesssim&\lambda_{N+1}^{-\gamma}\left(\frac{1}{\alpha\left(\epsilon\cdot\min\{H,1\}\right)^2}+\mathrm{E}\left[\left\|A^{\frac{\gamma}{2}}u(0)\right\|^2\right]\right)+\mathrm{E}\left[\left\|P_Nu(t)-u^N(t)\right\|^2\right].
\end{eqnarray*}
Then it is needed to estimate the bound of $\mathrm{E}\left[\left\|P_Nu(t)-u^N(t)\right\|^2\right]$. Let $e^N_t=P_Nu(t)-u^N(t)$. Combining Eqs.\ \eqref{eq:3.1-1} and  \eqref{eq:4.4-4} leads to
\begin{equation*}
 e^N_t=\int^t_0S_N(t-s)\left[f(u(s))-f(u^N(s))\right]\mathrm{d}s.
\end{equation*}
Then
\begin{equation*}
\frac{\mathrm{d}}{\mathrm{d}t}e^N_t=-A^\alpha_Ne^N_t+f_N(u(t))-f_N(u^N(t)),
\end{equation*}
which implies
\begin{eqnarray}\label{eq:4.6}
\frac{\mathrm{d}}{\mathrm{d}t}\left\|e^N_t\right\|^2&=&2\left\langle e^N_t,-A^\alpha_Ne^N_t+f\left(u(t)\right)-f\left(u^N(t)\right)\right\rangle\nonumber\\
&\lesssim&-\left\|A^{\frac{\alpha}{2}}_Ne^N_t\right\|^2+ \left\|e^N_t\right\|\left\|u(t)-u^N(t)\right\| \nonumber\\
&\lesssim& \left\|e^N_t\right\|\left\|u(t)-P_Nu(t)+P_Nu(t)-u^N(t)\right\| \nonumber\\
&\lesssim& \left\|e^N_t\right\|^2+\left\|e^N_t\|\|u(t)-P_Nu(t)\right\|\nonumber\\
&\lesssim& \left\|e^N_t\right\|^2+\left\|u(t)-P_Nu(t)\right\|^2;
\end{eqnarray}
the  H\"older inequality is used for the first inequality, and the fact that $ab\le \frac{a^2}{2C}+\frac{Cb^2}{2}$ is used for the fourth inequality.
Integrating Eq.\ \eqref{eq:4.6} from $0$ to $t$, from Theorem \ref{th:1} and Lemma \ref{le:2}, we have
\begin{eqnarray*}
\mathrm{E}\left[\left\|e^N_t\right\|^2\right]&\lesssim& \int^t_0\mathrm{E}\left[\left\|e^N_s\right\|^2+\left\|u(s)-P_Nu(s)\right\|^2\right]\mathrm{d}s\\
&\lesssim&\int^t_0\mathrm{E}\left[\left\|e^N_s\right\|^2\right]\mathrm{d}s+\lambda_{N+1}^{-\gamma}\left(\frac{1}{\alpha\left(\epsilon\cdot\min\{H,1\}\right)^2}+\mathrm{E}\left[\left\|A^{\frac{\gamma}{2}}u(0)\right\|^2\right]\right).
\end{eqnarray*}
By using the Gr\"onwall inequality, it has
\begin{equation*}
\mathrm{E}\left[\left\|e^N_t\right\|^2\right]\lesssim \lambda_{N+1}^{-\gamma}\left(\frac{1}{\alpha\left(\epsilon\cdot\min\{H,1\}\right)^2}+\mathrm{E}\left[\left\|A^{\frac{\gamma}{2}}u(0)\right\|^2\right]\right).
\end{equation*}
Then
\begin{equation*}
\left\|u(t)-u^N(t)\right\|_{L^2(D,U)}\lesssim \lambda_{N+1}^{-\frac{\gamma}{2}}\left(\frac{\alpha^{-\frac{1}{2}}}{\epsilon\cdot\min\{H,1\}}+\left\|u_0\right\|_{L^2\left(D,\dot{U}^\gamma\right)}\right).
\end{equation*}
\end{proof}
Note that if $N$ is big enough, taking $\epsilon=\frac{1}{\log \left(\lambda_{N+1}\right)}$ leads to
\begin{equation*}
\left\|u(t)-u^N(t)\right\|_{L^2(D,U)}\lesssim \lambda_{N+1}^{-\rho+\frac{1}{2}-\alpha\cdot\min\{H,1\}}\left(\frac{\alpha^{-\frac{1}{2}}\cdot\log \left(\lambda_{N+1}\right)}{\min\{H,1\}}+\left\|u_0\right\|_{L^2\left(D,\dot{U}^\gamma\right)}\right).
\end{equation*}

\section{Fully discrete scheme}
In this section, we are concerned with the time discretization of Eq.\ \eqref{eq:4.4}. Meanwhile, the error estimates for the fully discrete scheme are derived.

Using the semi-implicit Euler scheme, one can get the fully discrete scheme of Eq.\ \eqref{eq:3.1} as
\begin{equation}\label{eq:5.1-1}
u^{N,M}_{m+1}-u^{N,M}_{m}+\tau A^\alpha_N u^{N,M}_{m+1}=\tau f_N\left(u^{N,M}_{m}\right)+P_N\left(B_{H,\mu}(t_{m+1})-B_{H,\mu}(t_{m})\right),
\end{equation}
but Proposition \ref{le:1} implies that as $\alpha\ge \gamma$, $A^{\frac{\alpha}{2}}u(t)$ is not H\"older continuous, i.e.,
\begin{equation*}
\lim_{s\to t}\left\|A^{\frac{\alpha}{2}}\left(u(t)-u(s)\right)\right\|_{L^2(D,U)}\ne0. \end{equation*}
Then the approximation scheme \eqref{eq:5.1-1} is invalid. Therefore, we introduce the following technique to circumvent this defect.

Let $z^N(t)=u^N(t)-\int^t_0S_N(t-s)P_N\mathrm{d}B_{H,\mu}(s)$. If $u^N(t)$ is the unique mild solution of Eq.\ \eqref{eq:4.1}, then $z^N(t)$ is the unique mild solution of the following PDE
\begin{equation}\label{eq:5.1}
\frac{\mathrm{d}}{\mathrm{d}t}z^N(t)+A^\alpha_Nz^N(t)=f_N\left(u^N(t)\right)\mathrm{d}t\quad{\rm with}~~ t\in(0,T]\quad{\rm and}~~ z^N(0)=u^N(0).
\end{equation}
The unique mild solution of Eq.\ \eqref{eq:5.1} is given by
\begin{equation*}
z^N(t)=S_N(t)z^N(0)+\int^t_0S_N(t-s)f\left(u^N(s)\right)\mathrm{d}s.
\end{equation*}
If $\lim_{s\to t}\left\|A^{\frac{\alpha}{2}}_N\left(z(t)-z(s)\right)\right\|_{L^2(D,U)}=0$, then one can use Euler scheme to obtain the time discretization of Eq.\ \eqref{eq:5.1}. The following Theorem shows that $A^{\frac{\alpha}{2}}z(t)$ is H\"older continuous.
\begin{thm} \label{th:3}
Let Assumptions \ref{as:2.1}-\ref{as:2.2} be fulfilled and $z^N(t)$ be the mild solution of Eq.\ \eqref{eq:5.1}. Let $\left\|u(0)\right\|_{L^2(D,\dot{U}^{\gamma+\alpha})}<\infty$ and the conditions of Corollary \ref{cor:1} are also satisfied. Then we have
\begin{equation*}
\left\|A^{\frac{\alpha}{2}}_N\left(z^N(t)-z^N(s)\right)\right\|_{L^2(D,U)}\lesssim (t-s)^{\min\{\frac{\gamma}{2\alpha},1\}}\left(\frac{1}{\min\left\{\gamma\alpha H,\gamma^2\right\}}+\left\|u(0)\right\|_{L^2(D,\dot{U}^{\gamma+\alpha})}\right).
\end{equation*}
\end{thm}
\begin{proof}
By using the inequality $(a+b)^2\lesssim a^2+b^2$, we get
\begin{eqnarray}\label{eq:5.2-1}
&\mathrm{E}&\left[\left\|A^{\frac{\alpha}{2}}_N\left(z^N(t)-z^N(s)\right)\right\|^2\right]\nonumber\\
&\lesssim&\mathrm{E}\left[\left\|A^{\frac{\alpha}{2}}_N\left(S_N(t)-S_N(s)\right)z^N(0)\right\|^2\right]+\mathrm{E}\left[\left\|\int^t_sA^{\frac{\alpha}{2}}_NS_N(t-r)f\left(u^N(r)\right)\mathrm{d}r\right\|^2\right]\nonumber\\
&+&\mathrm{E}\left[\left\|\int^s_0A^{\frac{\alpha}{2}}_N\left(S_N(t-r)-S_N(s-r)\right)f\left(u^N(r)\right)\mathrm{d}r\right\|^2\right]\nonumber\\
&=&\widetilde{I}_1+\widetilde{I}_2+\widetilde{I}_3.
\end{eqnarray}
As $2\alpha > \gamma$, the inequality $\mathrm{e}^{-x}-\mathrm{e}^{-y}\lesssim|x-y|^{\theta_1}\,(0\le\theta_1\le1, x\ge0, y\ge0)$ implies
\begin{eqnarray}\label{eq:5.3-1}
\widetilde{I}_1&=&\mathrm{E}\left[\left\|\sum^N_{i=1}\lambda^{\frac{\alpha}{2}}_{i}\left(\mathrm{e}^{-\lambda^\alpha_{i}t}-\mathrm{e}^{-\lambda^\alpha_{i}s}\right)
\left\langle z^N(0),\phi_{i}(x)\right\rangle\phi_{i}(x)\right\|^2\right]\nonumber\\
&\lesssim&\mathrm{E}\left[\left\|\sum^N_{i=1}\lambda^{\frac{\alpha}{2}}_{i}\left(\left(\lambda^\alpha_{i}(t-s)\right)^{\frac{\gamma}{2\alpha}}\right)
\left\langle z^N(0),\phi_{i}(x)\right\rangle\phi_{i}(x)\right\|^2\right]\nonumber\\
&=&\mathrm{E}\left[\left\|(t-s)^{\frac{\gamma}{2\alpha}}A^{\frac{\gamma+\alpha}{2}}_Nz^N(0)\right\|^2\right]\nonumber\\
&\lesssim&(t-s)^{\frac{\gamma}{\alpha}}\mathrm{E}\left[\left\|A^{\frac{\gamma+\alpha}{2}}_Nu^N(0)\right\|^2\right].
\end{eqnarray}
Combining the inequality $\mathrm{e}^{-x}\lesssim x^{-\theta_1}\,(\theta_1\ge0,x\ge0)$ and Eq. \eqref{eq:sec3-1} leads to
\begin{eqnarray}\label{eq:5.4-1}
\widetilde{I}_2&=&\mathrm{E}\left[\left\|\int^t_s\sum^N_{i=1}\lambda^{\frac{\alpha}{2}}_{i}\mathrm{e}^{-\lambda^\alpha_{i}(t-r)}\left\langle f\left(u^N(r)\right),\phi_{i}(x)\right\rangle\phi_{i}(x)\mathrm{d}r\right\|^2\right]\nonumber\\
&\lesssim&\mathrm{E}\left[\left\|\int^t_s\sum^N_{i=1}\lambda^{\frac{\alpha}{2}}_{i}\left(\lambda^\alpha_{i}(t-r)\right)^{-\frac{\alpha-\frac{\gamma}{2}}{2\alpha}}\left\langle f\left(u^N(r)\right),\phi_{i}(x)\right\rangle\phi_{i}(x)\mathrm{d}r\right\|^2\right]\nonumber\\
&=&\mathrm{E}\left[\left\|\int^t_s(t-r)^{-\frac{\alpha-\frac{\gamma}{2}}{2\alpha}}A^{\frac{\gamma}{4}}f_N\left(u^N(r)\right)\mathrm{d}r\right\|^2\right]\nonumber\\
&\lesssim&\int^t_s\mathrm{E}\left[\left\|A^{\frac{\gamma}{4}}f_N\left(u^N(r)\right)\right\|^2\right]\mathrm{d}r\int^t_s(t-r)^{-\frac{\alpha-\frac{\gamma}{2}}{\alpha}}\mathrm{d}r\nonumber\\
&\lesssim&(t-s)^{1+\frac{\gamma}{2\alpha}}\left(\frac{1}{\min\left\{\gamma \alpha H,\gamma^2\right\}}+\mathrm{E}\left[\left\|u(0)\right\|^2\right]\right).
\end{eqnarray}
Similar to the derivation of Eqs.\ \eqref{eq:5.3-1} and \ \eqref{eq:5.4-1}, we have
\begin{eqnarray}\label{eq:5.5-1}
\widetilde{I}_3
&=&\mathrm{E}\left[\left\|\int^s_0\sum^N_{i=1}\lambda^{\frac{\alpha}{2}}_{i}
\left(\mathrm{e}^{-\lambda^\alpha_{i}(t-r)}-\mathrm{e}^{-\lambda^\alpha_{i}(s-r)}\right)\left\langle f\left(u^N(r)\right),\phi_{i}(x)\right\rangle\phi_{i}(x)\mathrm{d}r\right\|^2\right]\nonumber\\
&\lesssim&\mathrm{E}\left[\left\|\int^s_0A^{\frac{\alpha}{2}}_N\left(A^\alpha_N(s-r)\right)^{-\frac{1}{2}-\frac{\gamma}{4\alpha}}
\left(A^\alpha_N(t-s)\right)^{\frac{\gamma}{2\alpha}}f\left(u^N(r)\right)\mathrm{d}r\right\|^2\right]\nonumber\\
&\lesssim&(t-s)^{\frac{\gamma}{\alpha}}\left(\frac{1}{\min\left\{\gamma \alpha H,\gamma^2\right\}}+\mathrm{E}\left[\left\|u(0)\right\|^2\right]\right).
\end{eqnarray}
Combining  Eqs.\ \eqref{eq:5.3-1}, \ \eqref{eq:5.4-1}, and \ \eqref{eq:5.5-1} leads to
\begin{equation}\label{eq:5.6-1}
\mathrm{E}\left[\left\|A^{\frac{\alpha}{2}}_N\left(z^N(t)-z^N(s)\right)\right\|^2\right]\lesssim(t-s)^{\frac{\gamma}{\alpha}}\left(\frac{1}{\min\left\{\gamma \alpha H,\gamma^2\right\}}+\mathrm{E}\left[\left\|u(0)\right\|^2\right]\right).
\end{equation}
As $2\alpha \le \gamma$,  by using same procedure, we have
\begin{eqnarray}\label{eq:5.7-1}
&\mathrm{E}&\left[\left\|A^{\frac{\alpha}{2}}_N\left(z^N(t)-z^N(s)\right)\right\|^2\right]\nonumber\\
&\lesssim&(t-s)^{2}\mathrm{E}\left[\left\|A^{\alpha}_Nu^N(0)\right\|^2\right]+\mathrm{E}\left[\left\|\int^t_sA^{\frac{\alpha}{2}}_NS_N(t-r)f\left(u^N(r)\right)\mathrm{d}r\right\|^2\right]\nonumber\\
&+&\mathrm{E}\left[\left\|\int^s_0A^{\frac{\alpha}{2}}_N\left(S_N(s-r)\right)
-\left(S_N(t-s)\right)f\left(u^N(r)\right)\mathrm{d}r\right\|^2\right]\nonumber\\
&\lesssim&(t-s)^{2}\mathrm{E}\left[\left\|A^{\alpha}_Nu^N(0)\right\|^2\right]+\mathrm{E}\left[\left\|\int^t_sA^{\frac{\gamma}{4}}_Nf\left(u^N(r)\right)\mathrm{d}r\right\|^2\right]\nonumber\\
&+&(t-s)^{2}\mathrm{E}\left[\left\|\int^s_0(s-r)^{-\frac{5}{6}}A ^{\frac{2\alpha}{3}}_Nf\left(u^N(r)\right)\mathrm{d}r\right\|^2\right]\nonumber\\
&\lesssim&(t-s)^{2}\left(\frac{1}{\min\left\{\gamma\alpha H,\gamma^2\right\}}+\mathrm{E}\left[\left\|A^{\alpha}_Nu^N(0)\right\|^2\right]\right).
\end{eqnarray}
Combining Eqs.\ \eqref{eq:5.6-1} and \ \eqref{eq:5.7-1} results in
\begin{equation*}
\mathrm{E}\left[\left\|A^{\frac{\alpha}{2}}_N\left(z^N(t)-z^N(s)\right)\right\|^2\right]\lesssim (t-s)^{\min\{\frac{\gamma}{\alpha},2\}}\left(\frac{1}{\min\left\{\gamma\alpha H,\gamma^2\right\}}+\mathrm{E}\left[\left\|A^{\frac{\gamma+\alpha}{2}}_Nu^N(0)\right\|^2\right]\right).
\end{equation*}
\end{proof}

For time discretization of Eq.\ \eqref{eq:5.1}, we apply the classical semi-implicit Euler scheme. Let $\tau=\frac{T}{M}$ and $t_m=m\tau$ with $m=0,1,\dots,M$. Then one can obtain an approximation $z^{N,M}_m$ of $z^N(t_m)$ by the recurrence
\begin{equation}\label{eq:5.2}
z^{N,M}_{m+1}-z^{N,M}_{m}+\tau A^\alpha_N z^{N,M}_{m+1}=\tau f_N\left(u^{N,M}_{m}\right),
\end{equation}
where $u^{N,M}_{m}=z^{N,M}_{m}+\int^{t_m}_0S_N(t_m-s)P_N\mathrm{d}B_{H,\mu}(s)$.
The approximation of $\int^{t_m}_0S_N(t_m-s)P_N\mathrm{d}B_{H,\mu}(s)$ is given as
\begin{equation*}
\int^{t_m}_0S_N(t_m-s)P_N\mathrm{d}B_{H,\mu}(s)\approx
\sum_{k=0}^{t_m/\widetilde{\tau}-1}S_N(t_m-k\widetilde{\tau})\left(B_{H,\mu}(\left(k+1\right)\widetilde{\tau})-B_{H,\mu}(k\widetilde{\tau})\right).
\end{equation*}
Then
\begin{equation}\label{eq:5.7-7}
u^{N,M}_{m}=z^{N,M}_{m}+\sum_{k=0}^{t_m/\widetilde{\tau}-1}S_N(t_m-k\widetilde{\tau})\left(B_{H,\mu}(\left(k+1\right)\widetilde{\tau})-B_{H,\mu}(k\widetilde{\tau})\right).
\end{equation}

For the sake of completeness, we need to derive the error estimates for the approximation of $\int^{t_m}_0S_N(t_m-s)P_N\mathrm{d}B_{H,\mu}(s)$. 

\begin{prop}\label{le:3}
Under the conditions of Proposition \ref{le:1}, we have
\begin{eqnarray*}
&&\left\|\sum_{k=0}^{t_m/\widetilde{\tau}-1}\int^{\left(k+1\right)\widetilde{\tau}}_{k\widetilde{\tau}}\left(S_N(t_m-s)-S_N(t_m-k\widetilde{\tau})\right)\mathrm{d}B_{H,\mu}(s)\right\|_{L^2(D,U)}\nonumber\\
&&\lesssim
\left\{
\begin{array}{ll}
(2\rho H-H)^{-\frac{1}{2}}\widetilde{\tau}^{H}, &\ \rho>\frac{1}{2},\\
\left(\epsilon\alpha H\gamma\right)^{-\frac{1}{2}}\widetilde{\tau}^{H-\frac{1-2\rho+\epsilon\alpha\cdot\min\{H,1\}}{2\alpha}}, &\ 0<\rho\le\frac{1}{2}.
\end{array}
\right.
\end{eqnarray*}
\end{prop}
\begin{proof}
As $0<H<\frac{1}{2}$, Eq.\ \eqref{eq:2.3} and the triangle inequality imply that
\begin{eqnarray}\label{eq:5.9-1}
&&\left\|\sum_{k=0}^{t_m/\widetilde{\tau}-1}\int^{\left(k+1\right)\widetilde{\tau}}_{k\widetilde{\tau}}\left(S_N(t_m-s)-S_N(t_m-k\widetilde{\tau})\right)\mathrm{d}B_{H,\mu}(s)\right\|_{L^2(D,U)}\nonumber\\
&\lesssim&\sum_{k=0}^{t_m/\widetilde{\tau}-1}\left(\mathrm{E}\left[\left\|\sum^N_{i=1}\int^{\left(k+1\right)\widetilde{\tau}}_{k\widetilde{\tau}}\left(\mathrm{e}^{-(t_m-s)\lambda^\alpha_{i}}-\mathrm{e}^{-(t_m-k\widetilde{\tau})\lambda^\alpha_{i}}\right)\sigma_{i} \phi_{i}(x)\mathrm{d}\beta^{i}(s)\right\|^2\right]\right)^{\frac{1}{2}}\nonumber\\
&+&\sum_{k=0}^{t_m/\widetilde{\tau}-1}\left(\mathrm{E}\left[\left\|\sum^N_{i=1}\int^{\left(k+1\right)\widetilde{\tau}}_{k\widetilde{\tau}}\int^{\left(k+1\right)\widetilde{\tau}}_s
\left(\mathrm{e}^{-(t_m-u)\lambda^\alpha_{i}}\right.\right.\right.\right.\nonumber\\
&&\left.\left.\left.\left.-\mathrm{e}^{-(t_m-k\widetilde{\tau})\lambda^\alpha_{i}}\right)(u-s)^{H-\frac{1}{2}}
\mathrm{e}^{-\mu(u-s)}\mathrm{d}u\sigma_{i}\mu \phi_{i}(x)\mathrm{d}\beta^{i}(s)\right\|^2\right]\right)^{\frac{1}{2}}\nonumber\\
&+&\sum_{k=0}^{t_m/\widetilde{\tau}-1}\left(\mathrm{E}\left[\left\|\sum^N_{i=1}\int^{\left(k+1\right)\widetilde{\tau}}_{k\widetilde{\tau}}\int^{\left(k+1\right)\widetilde{\tau}}_s
\left(\mathrm{e}^{-(t_m-s)\lambda^\alpha_{i}}\right.\right.\right.\right.\nonumber\\
&&\left.\left.\left.\left.-\mathrm{e}^{-(t_m-u)\lambda^\alpha_{i}}\right)(u-s)^{H-\frac{3}{2}}
\mathrm{e}^{-\mu(u-s)}\mathrm{d}u\sigma_{i} \phi_{i}(x)\mathrm{d}\beta^{i}(s)\right\|^2\right]\right)^{\frac{1}{2}}\nonumber\\
&=&I+II+III.
\end{eqnarray}
Let $0\le\theta< 2\cdot\min\{H,1\}$. It$\mathrm{\hat{o}}$'s isometry leads to
\begin{eqnarray}\label{eq:5.9-2}
I&\lesssim&\sum_{k=0}^{t_m/\widetilde{\tau}-1}\left(\mathrm{e}^{-\frac{t_m-(k+1)\widetilde{\tau}}{2}\lambda^\alpha_{i}}-\mathrm{e}^{-\frac{t_m-k\widetilde{\tau}}{2}\lambda^\alpha_{i}}\right)\left(\sum^N_{i=1}\int^{\left(k+1\right)\widetilde{\tau}}_{k\widetilde{\tau}}\left(\mathrm{e}^{-\frac{t_m-s}{2}\lambda^\alpha_{i}}\right)^2\lambda_{i}^{-2\rho} \mathrm{d}s\right)^{\frac{1}{2}}\nonumber\\
&\lesssim&\sum_{k=0}^{t_m/\widetilde{\tau}-1}\left(\mathrm{e}^{-\frac{t_m-(k+1)\widetilde{\tau}}{2}\lambda^\alpha_{i}}-\mathrm{e}^{-\frac{t_m-k\widetilde{\tau}}{2}\lambda^\alpha_{i}}\right)\left(\sum^N_{i=1}\int^{\left(k+1\right)\widetilde{\tau}}_{k\widetilde{\tau}}(t_m-s)^{-\theta}\lambda_{i}^{-2\rho-\alpha\theta} \mathrm{d}s\right)^{\frac{1}{2}}\nonumber\\
&\lesssim&\left(\sum^N_{i=1}\lambda_{i}^{-2\rho-\alpha\theta} \widetilde{\tau}^{1-\theta}\right)^{\frac{1}{2}}
\end{eqnarray}
and
\begin{eqnarray}\label{eq:5.9-4}
II&\lesssim&\sum_{k=0}^{t_m/\widetilde{\tau}-1}\left(\mathrm{e}^{-\frac{t_m-(k+1)\widetilde{\tau}}{2}\lambda^\alpha_{i}}\right.\nonumber\\
&&-\left.\mathrm{e}^{-\frac{t_m-(k-1)\widetilde{\tau}}{2}\lambda^\alpha_{i}}\right)\left(\sum^N_{i=1}\lambda_i^{-2\rho}\int^{\left(k+1\right)\widetilde{\tau}}_{k\widetilde{\tau}}\left(\int^{\left(k+1\right)\widetilde{\tau}}_s
\mathrm{e}^{-\frac{t_m-u}{2}\lambda^\alpha_{i}}(u-s)^{H-\frac{1}{2}}
\mathrm{d}u \right)^2\mathrm{d}s\right)^{\frac{1}{2}}\nonumber\\
&\lesssim&\sum_{k=0}^{t_m/\widetilde{\tau}-1}\left(\mathrm{e}^{-\frac{t_m-(k+1)\widetilde{\tau}}{2}\lambda^\alpha_{i}}\right.\nonumber\\
&&\left.-\mathrm{e}^{-\frac{t_m-(k-1)\widetilde{\tau}}{2}\lambda^\alpha_{i}}\right)\left(\sum^N_{i=1}\lambda_i^{-2\rho}\lambda^{-\alpha\theta}_{i}\int^{\left(k+1\right)\widetilde{\tau}}_{k\widetilde{\tau}}(\left(k+1\right)\widetilde{\tau}-s)^{2H+1-\theta}\mathrm{d}s\right)^{\frac{1}{2}}\nonumber\\
&\lesssim&\left(\sum^N_{i=1}\lambda_i^{-2\rho-\alpha\theta}\widetilde{\tau}^{2H+2-\theta}\right)^{\frac{1}{2}}.
\end{eqnarray}
Similarly, for $III$, let $0\le\theta=2\eta-2\delta< 2H$ and $\delta>\frac{1}{2}-H$. We have
\begin{eqnarray}\label{eq:5.9-3}
III&\lesssim&\left(\sum^N_{i=1}\lambda_i^{-2\rho}\int^{t_m}_{t_m-\widetilde{\tau}}\left(\int^{t_m}_s\left(\mathrm{e}^{-(t_m-s)\lambda^\alpha_{i}}-\mathrm{e}^{-(t_m-u)\lambda^\alpha_{i}}\right)(u-s)^{H-\frac{3}{2}}
\mathrm{d}u\right)^2\mathrm{d}s\right)^{\frac{1}{2}}\nonumber\\
&&+\sum_{k=0}^{t_m/\widetilde{\tau}-2}\left(\sum^N_{i=1}\int^{\left(k+1\right)\widetilde{\tau}}_{k\widetilde{\tau}}\lambda_i^{-2\rho}\mathrm{e}^{-(t_m-s)\lambda^\alpha_{i}}\left(\int^{\left(k+1\right)\widetilde{\tau}}_s
\left(\mathrm{e}^{-\frac{t_m-s}{2}\lambda^\alpha_{i}}\right.\right.\right.\nonumber\\
&&-\left.\left.\left.\mathrm{e}^{-\left(\frac{t_m+s}{2}-u\right)\lambda^\alpha_{i}}\right)(u-s)^{H-\frac{3}{2}}\mathrm{d}u\right)^2\mathrm{d}s\right)^{\frac{1}{2}}\nonumber\\
&\lesssim&\left(\sum^N_{i=1}\lambda_i^{-2\rho}\int^{t_m}_{t_m-\widetilde{\tau}}\lambda^{-2\alpha(\eta-\delta)}_{i}(t_m-s)^{2H-2\eta+2\delta-1}
\mathrm{d}s\right)^{\frac{1}{2}}\nonumber\\
&&+\sum_{k=0}^{t_m/\widetilde{\tau}-2}\left(\sum^N_{i=1}\int^{\left(k+1\right)\widetilde{\tau}}_{k\widetilde{\tau}}\lambda_i^{-2\rho}\mathrm{e}^{-(t_m-s)\lambda^\alpha_{i}}\left(\mathrm{e}^{-\frac{t_m-(k+2)\widetilde{\tau}}{2}\lambda^\alpha_{i}}\right.\right.\nonumber\\
&&-\left.\left.\mathrm{e}^{-\frac{t_m-k\widetilde{\tau}}{2}\lambda^\alpha_{i}}\right)
\left(\int^{\left(k+1\right)\widetilde{\tau}}_s
\lambda^{\frac{\alpha(1-\theta)}{2}}_{i}(u-s)^{H-1-\frac{\theta}{2}}\mathrm{d}u\right)^2\mathrm{d}s\right)^{\frac{1}{2}}\nonumber\\
&\lesssim&\left(\sum^N_{i=1}\lambda_i^{-2\rho-\alpha\theta}
\widetilde{\tau}^{2H-\theta}\right)^{\frac{1}{2}}+\sum_{k=0}^{t_m/\widetilde{\tau}-2}\left(\sum^N_{i=1}\int^{\left(k+1\right)\widetilde{\tau}}_{k\widetilde{\tau}}\lambda_i^{-2\rho}\mathrm{e}^{-(t_m-s)\lambda^\alpha_{i}}\mathrm{d}s\left(\mathrm{e}^{-\frac{t_m-(k+2)\widetilde{\tau}}{2}\lambda^\alpha_{i}}\right.\right.\nonumber\\
&&-\left.\left.\mathrm{e}^{-\frac{t_m-k\widetilde{\tau}}{2}\lambda^\alpha_{i}}\right)
\lambda^{\alpha(1-\theta)}_{i}\widetilde{\tau}^{2H-\theta}\right)^{\frac{1}{2}}\nonumber\\
&\lesssim&\sum_{k=0}^{t_m/\widetilde{\tau}-2}\left(\mathrm{e}^{-\frac{t_m-(k+2)\widetilde{\tau}}{2}\lambda^\alpha_{i}}-\mathrm{e}^{-\frac{t_m-k\widetilde{\tau}}{2}\lambda^\alpha_{i}}\right)\left(\sum^N_{i=1}\lambda_i^{-2\rho-\alpha\theta}
\widetilde{\tau}^{2H-\theta}\right)^{\frac{1}{2}}\nonumber\\
&\lesssim&\left(\sum^N_{i=1}\lambda_i^{-2\rho-\alpha\theta}
\widetilde{\tau}^{2H-\theta}\right)^{\frac{1}{2}}.
\end{eqnarray}
Combining Eqs. \eqref{eq:sec3-2} and  \eqref{eq:3.11} results in
\begin{equation*}
I+II+III\lesssim
\left\{
\begin{array}{ll}
(2\rho H-H)^{-\frac{1}{2}}\widetilde{\tau}^{H}, &\ \rho>\frac{1}{2}, \theta=0,\\
\left(\epsilon\alpha H\gamma\right)^{-\frac{1}{2}}\widetilde{\tau}^{H-\frac{1-2\rho+\epsilon\alpha\cdot\min\{H,1\}}{2\alpha}}, &\ 0<\rho\le\frac{1}{2}, \theta=\frac{1-2\rho+\epsilon\alpha\cdot\min\{H,1\}}{\alpha}.
\end{array}
\right.
\end{equation*}

As $H>\frac{1}{2}$ and $0\le\theta<2\cdot\min\{H,1\}$, similar to the derivation of Eq.\ \eqref{eq:5.9-4}, using Eqs. \eqref{eq:sec3-2} and \eqref{eq:3.11} leads to
\begin{eqnarray*}
&&\left\|\sum_{k=0}^{t_m/\widetilde{\tau}-1}\int^{\left(k+1\right)\widetilde{\tau}}_{k\widetilde{\tau}}\left(S_N(t_m-s)-S_N(t_m-k\widetilde{\tau})\right)\mathrm{d}B_{H,\mu}(s)\right\|_{L^2(D,U)}\nonumber\\
&&\lesssim
\left\{
\begin{array}{ll}
(2\rho H-H)^{-\frac{1}{2}}\widetilde{\tau}^{H}, &\ \rho>\frac{1}{2},\\
\left(\epsilon\alpha H\gamma\right)^{-\frac{1}{2}}\widetilde{\tau}^{H-\frac{1-2\rho+\epsilon\alpha\cdot\min\{H,1\}}{2\alpha}}, &\ 0<\rho\le\frac{1}{2}.
\end{array}
\right.
\end{eqnarray*}
\end{proof}

The following Theorem shows the convergence rates of time discretization.

\begin{thm} \label{th:4}
Let $u^N(t)$ be the mild solution of Eq.\ \eqref{eq:4.4-4}. Suppose that Assumptions \ref{as:2.1}-\ref{as:2.2} are satisfied, $\left\|u(0)\right\|_{L^2(D,\dot{U}^{\gamma+\alpha})}<\infty$, and $\epsilon=\frac{1}{\left|\log\tau\right|}$. Then

$\mathrm{(i)}$ For $\rho>\frac{1}{2}$,
\begin{equation*}
\left\|u^N(t_m)-u^{N,M}_{m}\right\|_{L^2(D,U)}\lesssim \tau^{\min\{H,1\}}\left(\frac{1}{\min\{\gamma, \alpha H,2\rho-1\}}+\left\|u_0\right\|_{L^2\left(D,\dot{U}^\gamma\right)}\right);
\end{equation*}

$\mathrm{(ii)}$ For $0<\rho\le\frac{1}{2}$,

\begin{eqnarray*}
\left\|u^N(t_m)-u^{N,M}_{m}\right\|_{L^2(D,U)}\lesssim\tau^{\frac{2\rho-1+2\alpha\cdot\min\{H,1\} }{2\alpha}}\left(\frac{\left|\log\tau\right|}{\min\{H,1\}}+\left\|u_0\right\|_{L^2\left(D,\dot{U}^\gamma\right)}\right).
\end{eqnarray*}

\end{thm}

\begin{proof}
Combining Eq.\ \eqref{eq:5.7-7} and Proposition \ref{le:3}, there exists
\begin{eqnarray*}
\left\|u^N(t_m)-u^{N,M}_{m}\right\|_{L^2(D,U)}&\lesssim&\left\|z^N(t_m)-z^{N,M}_{m}\right\|_{L^2(D,U)}\\
&+&\left\|\sum_{k=0}^{t_m/\widetilde{\tau}-1}\int^{\left(k+1\right)\widetilde{\tau}}_{k\widetilde{\tau}}\left(S_N(t_m-s)-S_N(t_m-k\widetilde{\tau})\right)\mathrm{d}B_{H,\mu}(s)\right\|_{L^2(D,U)}.
\end{eqnarray*}
Thus, we just need to estimate the bound of $\left\|z^N(t_m)-z^{N,M}_{m}\right\|_{L^2(D,U)}$. Let $e_m=z^N(t_m)-z^{N,M}_{m}$ and $\chi\in U$. From Eqs.\ \eqref{eq:5.1} and \ \eqref{eq:5.2}, we have
\begin{eqnarray*}
\left\langle e_{m+1}-e_m,\chi\right\rangle&=&-\int^{t_{m+1}}_{t_m}\left\langle\left(A^\alpha_Nz^N(s)-A^\alpha_Nz^{N,M}_{m+1}\right),\chi\right\rangle\mathrm{d}s\\
&&+\int^{t_{m+1}}_{t_m}\left\langle\left(f_N\left(u^N(s)\right)-f_N\left(u^{N,M}_{m}\right)\right),\chi\right\rangle\mathrm{d}s.
\end{eqnarray*}
Set $\chi=e_{m+1}$. Using the fact $(a-b)a=\frac{1}{2}(a^2-b^2)+\frac{1}{2}(a-b)^2$ in the left-hand side of the above equation, we get
\begin{eqnarray}\label{eq:5.3}
&&\frac{1}{2}\left(\mathrm{E}\left[\left\|e_{m+1}\right\|^2\right]-\mathrm{E}\left[\left\|e_{m}\right\|^2\right]\right)
+\frac{1}{2}\mathrm{E}\left[\left\|e_{m+1}-e_m\right\|^2\right]\nonumber\\
&&=\mathrm{E}\left[-\int^{t_{m+1}}_{t_m}\left\langle \left(A^\alpha_Nz^N(s)-A^\alpha_Nz^{N,M}_{m+1}\right), e_{m+1}\right\rangle\mathrm{d}s\right]\nonumber\\
&&~~~~ +\mathrm{E}\left[\int^{t_{m+1}}_{t_m}\left\langle f_N\left(u^N(s)\right)-f_N\left(u^{N,M}_{m}\right),e_{m+1}\right\rangle\mathrm{d}s\right].
\end{eqnarray}
To obtain the estimate of $\mathrm{E}\left[\left\|e_{m}\right\|^2\right]$, we need to bound the right-hand side of Eq.\ \eqref{eq:5.3}.

As $\rho>\frac{1}{2}$, from Theorem \ref{th:3}, it has
\begin{eqnarray}\label{eq:5.4}
&\mathrm{E}&\left[-\int^{t_{m+1}}_{t_m}\left\langle \left(A^\alpha_Nz^N(s)-A^\alpha_Nz^{N,M}_{m+1}\right), e_{m+1}\right\rangle\mathrm{d}s\right]\nonumber\\
&=&-\mathrm{E}\left[\int^{t_{m+1}}_{t_m}\sum^N_{i=1}\left\langle \lambda^{\frac{\alpha}{2}}_{i}\left(z^N(s)-z^{N,M}_{m+1}\right), \lambda^{\frac{\alpha}{2}}_{i}e_{m+1}\right\rangle\mathrm{d}s\right]\nonumber\\
&=&-\mathrm{E}\left[\int^{t_{m+1}}_{t_m}\left\langle A^{\frac{\alpha}{2}}_N\left(z^N(s)-z^{N,M}_{m+1}\right), A^{\frac{\alpha}{2}}_Ne_{m+1}\right\rangle\mathrm{d}s\right]\nonumber\\
&=&-\mathrm{E}\left[\int^{t_{m+1}}_{t_m}\left\langle A^{\frac{\alpha}{2}}_N\left(z^N(s)-z^N(t_{m+1})\right), A^{\frac{\alpha}{2}}_Ne_{m+1}\right\rangle\mathrm{d}s\right]-\mathrm{E}\left[\tau\left\|A^{\frac{\alpha}{2}}_Ne_{m+1}\right\|^2\right]\nonumber\\
&\lesssim&\mathrm{E}\left[\int^{t_{m+1}}_{t_m}\left\| A^{\frac{\alpha}{2}}_N\left(z^N(s)-z^N(t_{m+1})\right)\right\| \left\|A^{\frac{\alpha}{2}}_Ne_{m+1}\right\|\mathrm{d}s\right]-\mathrm{E}\left[\tau\left\|A^{\frac{\alpha}{2}}_Ne_{m+1}\right\|^2\right]\nonumber\\
&\lesssim&\frac{1}{2}\int^{t_{m+1}}_{t_m}\mathrm{E}\left[\left\|A^{\frac{\alpha}{2}}_N\left(z^N(s)-z^N(t_{m+1})\right)\right\|^2\right]\mathrm{d}s+\frac{1}{2}\mathrm{E}\left[\tau\left\|A^{\frac{\alpha}{2}}_Ne_{m+1}\right\|^2\right]-\mathrm{E}\left[\tau\left\|A^{\frac{\alpha}{2}}_Ne_{m+1}\right\|^2\right]\nonumber\\
&\lesssim&\frac{1}{2}\int^{t_{m+1}}_{t_m}\mathrm{E}\left[\left\|A^{\frac{\alpha}{2}}_N\left(z^N(s)-z^N(t_{m+1})\right)\right\|^2\right]\mathrm{d}s\nonumber\\
&\lesssim&\tau^{\min\{\frac{\gamma}{\alpha}+1,3\}}\left(\frac{1}{\min\left\{\gamma \alpha H,\gamma^2\right\}}+\mathrm{E}\left[\left\|A^{\frac{\gamma+\alpha}{2}}u_0\right\|^2\right]\right).
\end{eqnarray}
In the first and second inequalities, we use the H$\mathrm{\ddot{o}}$lder inequality and Young's inequality, respectively.
Combining the H$\mathrm{\ddot{o}}$lder inequality, Corollary \ref{cor:1}, and Proposition \ref{le:3} leads to
\begin{eqnarray}\label{eq:5.5}
&\mathrm{E}&\left[\int^{t_{m+1}}_{t_m}\left\langle f_N\left(u^N(s)\right)-f_N\left(u^{N,M}_{m}\right),e_{m+1}\right\rangle\mathrm{d}s\right]\nonumber\\
&=&\mathrm{E}\left[\int^{t_{m+1}}_{t_m}\left\langle f_N\left(u^N(s)\right)-f_N\left(u^N(t_m)\right),e_{m+1}\right\rangle\mathrm{d}s\right]\nonumber\\
&&+\mathrm{E}\left[\int^{t_{m+1}}_{t_m}\left\langle f_N\left(u^N(t_m)\right)-f_N\left(u^{N,M}_{m}\right),e_{m+1}\right\rangle\mathrm{d}s\right]\nonumber\\
&\lesssim&\mathrm{E}\left[\int^{t_{m+1}}_{t_m}\left\|u^N(s)-u^N(t_m)\right\|\left\|e_{m+1}\right\|\mathrm{d}s\right]+\mathrm{E}\left[\int^{t_{m+1}}_{t_m} \left\|u^N(t_m)-u^{N,M}_{m}\right\|\left\|e_{m+1}\right\|\mathrm{d}s\right]\nonumber\\
&\lesssim&\mathrm{E}\left[\int^{t_{m+1}}_{t_m}\left\|u^N(s)-u^N(t_m)\right\|^2\mathrm{d}s\right]+\tau\mathrm{E}\left[\left\|e_{m+1}\right\|^2+\left\|e_{m}\right\|^2\right]+\frac{\tau^{2H+1}}{2\rho H-H}\nonumber\\
&\lesssim&\tau^{1+\min\{2H,2\}}\left(\frac{1}{\min\{\gamma\alpha H, \alpha H^2,2\rho H-H\}}+\mathrm{E}\left[\left\|A^{\frac{\gamma}{2}}u_0\right\|^2\right]\right)\nonumber\\
&&+\tau\mathrm{E}\left[\left\|e_{m+1}\right\|^2+\left\|e_{m}\right\|^2\right].
\end{eqnarray}
Combining Eqs.\ \eqref{eq:5.3}, \ \eqref{eq:5.4}, and \ \eqref{eq:5.5}, we have
\begin{eqnarray}\label{eq:5.6}
&&\frac{1}{2}\left(\mathrm{E}\left[\left\|e_{m+1}\right\|^2\right]-\mathrm{E}\left[\left\|e_{m}\right\|^2\right]\right)\nonumber\\
&&\lesssim \tau^{1+\min\{2H,2\}}\left(\frac{1}{\min\{\gamma\alpha H, \alpha H^2,2\rho H-H\}}+\mathrm{E}\left[\left\|A^{\frac{\gamma+\alpha}{2}}u_0\right\|^2\right]\right)\nonumber\\
&&+\tau\mathrm{E}\left[\left\|e_{m+1}\right\|^2+\left\|e_{m}\right\|^2\right].
\end{eqnarray}
Summing $m$ in Eqs.\ \eqref{eq:5.6} from 0 to $\widetilde{m}~(0\le\widetilde{m}\le M-1)$ gives
\begin{eqnarray*}
\mathrm{E}\left[\left\|e_{\widetilde{m}+1}\right\|^2\right]&\lesssim & \tau^{\min\{2H,2\}}\left(\frac{1}{\min\{\gamma\alpha H, \alpha H^2,2\rho H-H\}}+\mathrm{E}\left[\left\|A^{\frac{\gamma+\alpha}{2}}u_0\right\|^2\right]\right)\nonumber\\
&&+\sum^{\widetilde{m}}_{m=0}\tau\mathrm{E}\left[\left\|e_{m+1}\right\|^2\right].
\end{eqnarray*}
By using the discrete Gr\"onwall inequality, we have
\begin{equation*}
\mathrm{E}\left[\left\|e_{m+1}\right\|^2\right]\lesssim\tau^{\min\{2H,2\}}\left(\frac{1}{\min\{\gamma\alpha H, \alpha H^2,2\rho H-H\}}+\mathrm{E}\left[\left\|A^{\frac{\gamma+\alpha}{2}}u_0\right\|^2\right]\right).
\end{equation*}
The case $\mathrm{(ii)}$ can be similarly proved.

\end{proof}

Combining Theorems \ref{th:2} and \ref{th:4}, we get the error bounds for the full discretization.
\begin{thm} \label{th:5}
Let Assumption \ref{as:2.1}-\ref{as:2.2} be satisfied and $u(t)$ be the mild solution of Eq.\ \eqref{eq:3.1}. If $\left\|u(0)\right\|_{L^2(D,\dot{U}^{\gamma+\alpha})}<\infty$, then we have

$\mathrm{(i)}$ For $\rho>\frac{1}{2}$,
\begin{eqnarray*}
\left\|u(t_m)-u^{N,M}_{m}\right\|_{L^2(D,U)}&\lesssim&\lambda_{N+1}^{-\frac{2\rho-1+2\alpha\cdot\min\{H,1\}}{2}}\left(\frac{\alpha^{-\frac{1}{2}}\log \lambda_{N+1}}{\min\{H,1\}}+\left\|u_0\right\|_{L^2\left(D,\dot{U}^\gamma\right)}\right)\\
&+&\tau^{\min\{H,1\}}\left(\frac{1}{\min\{\gamma, \alpha H,2\rho-1\}}+\left\|u_0\right\|_{L^2\left(D,\dot{U}^\gamma\right)}\right);
\end{eqnarray*}

$\mathrm{(ii)}$ For $0<\rho\le\frac{1}{2}$,

\begin{eqnarray*}
\left\|u(t_m)-u^{N,M}_{m}\right\|_{L^2(D,U)}&\lesssim&\lambda_{N+1}^{-\frac{2\rho-1+2\alpha\cdot\min\{H,1\}}{2}}\left(\frac{\alpha^{-\frac{1}{2}}\log \lambda_{N+1}}{\min\{H,1\}}+\left\|u_0\right\|_{L^2\left(D,\dot{U}^\gamma\right)}\right)\\
&+&\tau^{\frac{2\rho-1+2\alpha\cdot\min\{H,1\} }{2\alpha}}\left(\frac{\left|\log\tau\right|}{\min\{H,1\}}+\left\|u_0\right\|_{L^2\left(D,\dot{U}^\gamma\right)}\right).
\end{eqnarray*}

\end{thm}

\section{Numerical experiments}
In this section, we present the simulation results to show the convergence behaviour of the full discretization scheme and the effect of the parameters $H$, $\alpha$, and $\rho$ on the convergence rates. All numerical errors are given in the sense of mean-squared $L^2$-norm.

We solve \eqref{eq:1.1} in the two-dimensional domain $D=(0,1)\times(0,1)$ by the proposed method with  Dirichlet eigenpairs $\lambda_{i,j}=\pi^2\left(i^2+j^2\right)$, $\phi_{i,j}=2\sin(i\pi x_1)\sin(j\pi x_2)$, $x=(x_1,x_2)$, and $i,j=1,2,\dots, N$. The numerical results with a smooth initial data $u(x,0)=x_1^2x_2^2$ and $f(u(x,t))=u(x,t)$ are presented in tables, where $u^{N,M}_{m}$ denotes the numerical solution with fixed time step size $\tau=\frac{T}{M}$ at time $t=m \tau$. Since the exact solutions of Eq.\ \eqref{eq:1.1} are unknown, we use the following formulas to calculate the convergence rates:
\begin{eqnarray*}
 \textrm{convergence rate in space}&=&
 \frac{\ln\left(\left\|u^{1.5N,M}_{M}-u^{N,M}_{M}\right\|_{L^2(D,U)}/
\left\|u^{N,M}_{M}-u^{N/1.5,M}_{M}\right\|_{L^2(D,U)}\right)}{\ln1.5},\\
  \textrm{convergence rate in time}&=&
 \frac{\ln\left(\left\|u^{N,1.5M}_{1.5M}-u^{N,M}_{M}\right\|_{L^2(D,U)}/
\left\|u^{N,M}_{M}-u^{N,M/1.5}_{M/1.5}\right\|_{L^2(D,U)}\right)}{\ln1.5}.
\end{eqnarray*}

In the numerical simulations, the errors $\left\|u^{1.5N,M}_{M}-u^{N,M}_{M}\right\|_{L^2(D,U)}$ are calculated by Monte Carlo method, i.e.,
\begin{equation*}
\left\|u^{1.5N,M}_{M}-u^{N,M}_{M}\right\|_{L^2(D,U)}\approx \left(\frac{1}{K}\sum^K_{k=1}\left\|u^{1.5N,M}_{M,k}-u^{N,M}_{M,k}\right\|^2\right)^{\frac{1}{2}}.
\end{equation*}
We take $K=1000$ as the number of the simulation trajectories. The symbol $k$ represents the $k$-th trajectory.

\begin{table}[H]
\renewcommand\arraystretch{1.6}
\caption{Time convergence rates with $N=50$, $T=0.5$, $\alpha=0.5$, $\rho=0.75$, and $\lambda=1$}
\centering
\begin{tabular}{c c c c c c c c c}
\hline
$M$ & $H=0.4$ & Rate &$H=0.8$ &Rate & $H=1.2$ &Rate &$H=1.6$&Rate \\
\Xhline{1.2pt}
  32&8.895e-02&     & 8.241e-03&     &1.775e-03&     &9.924e-04&     \\
  48&7.578e-02& 0.395& 5.976e-03& 0.793&1.206e-03&0.952&6.540e-04&1.028 \\
  72&6.420e-02& 0.409& 4.314e-03& 0.804&8.118e-04&0.977&4.326e-04& 1.019 \\
  \hline
\end{tabular}
\end{table}
From Tables 1, one can see that the time convergence rates increases with the increase of $H$. When $0<H<1$, the proposed methods have $H$-order convergence in time. As $H\ge1$, the convergence rate is first-order convergence in time. The numerical results confirm the error estimate in Theorem \ref{th:5}.
\begin{table}[H]
\renewcommand\arraystretch{1.6}
\caption{Time convergence rates with $N=50$, $T=0.5$, $\alpha=0.8$, $\rho=0.4$, and $\lambda=0.5$}
\centering
\begin{tabular}{c c c c c c c c c}
\hline
$M$ & $H=0.6$ & Rate &$H=0.8$ &Rate & $H=1$ &Rate &$H=1.2$&Rate \\
\Xhline{1.2pt}
  32&6.407e-02 &     &1.996e-02 &     &6.811e-03&     &2.711e-03&     \\
  48&5.413e-02 & 0.416&1.553e-02 &0.619 &5.024e-03& 0.751 &1.902e-03& 0.874\\
  72&4.525e-02 & 0.442&1.200e-02&0.637 &3.625e-03& 0.805 &1.324e-03& 0.893\\
  \hline
\end{tabular}
\end{table}
 Tables 2 demonstrates that the time convergence rates decrease with the increase of $H$, when $0<\rho\le\frac{1}{2}$. The theoretical convergence rates are near to $\frac{2\rho-1+2\alpha\cdot\min\{H,1\} }{2\alpha}$ in time. The numerical simulation agrees well with the theoretical results.


 \begin{table}[H]
 \renewcommand\arraystretch{1.6}
\caption{Space convergence rates with $M =2000$, $T=0.25$, $\alpha=0.3$, $\rho=0.75$, and $\lambda=0.5$}
\centering
\begin{tabular}{c c c c c c c c c}
\hline
$N$ & $H=0.35$ & Rate &$H=0.7$ &Rate & $H=1.05$ &Rate &$H=1.4$&Rate \\
\Xhline{1.2pt}
  16&4.190e-02&     & 1.431e-02&     &8.352e-03&     &6.969e-03&     \\
  24&3.195e-02& 0.668& 1.009e-02& 0.862&5.525e-03& 1.019&4.515e-03&
  1.071\\
  36&2.402e-02& 0.703& 6.967e-03& 0.913&3.619e-03& 1.044&2.924e-03& 1.072\\
  \hline
\end{tabular}
\end{table}
Tables 3 shows that the space convergence rates increase with the increase of $H$ as $H\le1$; and the convergence rates tend to $1.1$ as $H>1$. The numerical results confirm the theoretical prediction of convergence rates close to $2\alpha\cdot\min\{H,1\}+2\rho-1$, given in Theorem \ref{th:5}.


 \begin{table}[H]
 \renewcommand\arraystretch{1.6}
\caption{Space convergence rates with $M =2000$, $T=0.25$, $H=0.8$, $\rho=0.75$ and $\lambda=0.5$}
\centering
\begin{tabular}{c c c c c c c c c}
\hline
$N$ & $\alpha=0.2$ & Rate &$\alpha=0.4$ &Rate & $\alpha=0.6$ &Rate &$\alpha=0.8$&Rate  \\
\Xhline{1.2pt}
  16&2.450e-02&     & 5.997e-03&     &1.637e-03&     &4.306e-04&     \\
  24&1.793e-02&0.770 & 3.938e-03& 1.037&9.389e-04& 1.371&2.106e-04& 1.764\\
  36&1.306e-02& 0.782& 2.532e-03& 1.090&5.287e-04& 1.416&1.021e-04& 1.786\\
  \hline
\end{tabular}
\end{table}
As $\alpha=0.2$, $0.4$, $0.6$, and $0.8$, the theoretical convergence rates in space are near to $0.82$, $1.14$, $1.46$, and $1.78$, respectively. The numerical results obey the theoretical prediction and show that the space convergence rates increase with the increase of $\alpha$ in Tables 4.

\begin{table}[H]
 \renewcommand\arraystretch{1.6}
\caption{Space convergence rates with $M =4000$, $T=0.25$, $H=1.2$, $\alpha=0.5$, and $\lambda=0.5$}
\centering
\begin{tabular}{c c c c c c c c c}
\hline
$N$ & $\rho=0.75$ & Rate &$\rho=1.25$ &Rate & $\rho=1.75$ &Rate &$\rho=2.25$&Rate  \\
\Xhline{1.2pt}
  16&1.446e-03&     & 5.964e-05&     &2.665e-06
&     &1.279e-07
&     \\
  24&8.185e-04&1.4035 & 2.284e-05& 2.368&6.870e-07
& 3.343& 2.173e-08
& 4.372\\
  36&4.581e-04& 1.4314& 8.649e-06& 2.395&1.762e-07
& 3.356&3.635e-09& 4.410\\
  \hline
\end{tabular}
\end{table}
Tables 5 shows that the space convergence rates increase with the increase of $\rho$. That is, the space convergence rates of the spectral Galerkin method can be continuously enhanced by improving the regularity of the mild solution in space.

\section{Conclusion}
In this paper, we attempt to investigate the regularity of mild solution and the numerical approximation for the  fractional stochastic PDEs, modelling the subordinated killed Brownian motion, driven by a tempered fractional Gaussian noise. When the regularity of the initial value is good enough, the regularity of mild solution depends on the infinite dimensional stochastic integration $\int^t_0S(t-s)\mathrm{d}B_{H,\mu}(s)$ and $\alpha$. Notice that the regularity is improved in space, when $\alpha$ increases, but the H\"older regularity is reduced in time. For numerical approximation, the spectral Galerkin method is suitable for space discretization of \eqref{eq:4.4}, since $u(x,t)$ belongs to $L^2(D,\dot{U}^\gamma)$. And the spectral Galerkin method can achieve a better convergence rate with the increase of $\gamma$. The H\"older continuity of mild solution and numerical scheme dominate the convergence rate in time. Unlike the PDEs, as $\alpha\ge\gamma$, the infinite dimensional stochastic integration leads to $\lim_{s\to t}\left\|A^{\frac{\alpha}{2}}\left(u(t)-u(s)\right)\right\|_{L^2(D,U)}\ne0$, which makes the semi-implicit Euler scheme not work for \eqref{eq:4.4}.
To obtain an effective semi-implicit Euler scheme for any $\alpha$, improving the H\"older regularity of mild solution is a simple and effective strategy. By transforming spectral Galerkin approximation Eq.\ \eqref{eq:4.4} into an equivalent form Eq.\ \eqref{eq:5.1}, the better H\"older regularity of mild solution is obtained.
 Using the semi-implicit Euler scheme to discretize Eq.\ \eqref{eq:5.1} in time, we obtain the strong convergence rates, which are less than or equal to $1$.


\section*{Acknowledgements}
This work was supported by the National Natural Science Foundation of China under Grant
No. 11671182, and the Fundamental Research Funds for the Central Universities under Grants No. lzujbky-2018-ot03.

\appendix
\section{Proof of the uniqueness for the solution of Eq.\ \eqref{eq:3.1-1}}
Let $u(t,\omega)$ and $\hat{u}(t,\omega)$ be the solutions with initial values $u(0,\omega)$ and $\hat{u}(0,\omega)$, respectively, where $\omega$ denotes the path of stochastic process.
By using Eq.\ \eqref{eq:3.1-1}, we have
\begin{eqnarray*}
\mathrm{E}\left[\left\|u(t)-\hat{u}(t)\right\|^2\right]&=&\mathrm{E}\left[\left\|S(t)u(0)-S(t)\hat{u}(0)
+\int^t_0S(t-s)f(u(s))\mathrm{d}s-\int^t_0S(t-s)f(\hat{u}(s))\mathrm{d}s\right\|^2\right]\\
&\le&2\mathrm{E}\left[\left\|u(0)-\hat{u}(0)\right\|^2\right]+C\int^t_0\mathrm{E}\left[\left\|u(s)-\hat{u}(s)\right\|^2\right]\mathrm{d}s.
\end{eqnarray*}
The Gr\"onwall inequality leads to
\begin{equation*}
\mathrm{E}\left[\left\|u(t)-\hat{u}(t)\right\|^2\right]\le 2\mathrm{E}\left[\left\|u(0)-\hat{u}(0)\right\|^2\right]\exp(Ct).
\end{equation*}
Taking $u(0)=\hat{u}(0)$ leads to
\begin{equation*}
\mathrm{E}\left[\left\|u(t)-\hat{u}(t)\right\|^2\right]=0,  \quad t\in[0,T].
\end{equation*}
From the continuity of $t\to \left\|u(t)-\hat{u}(t)\right\|^2$, we get
\begin{equation*}
P\left[\left\|u(t,\omega)-\hat{u}(t,\omega)\right\|^2=0, \quad t\in[0,T]\right]=1,
\end{equation*}
where $P$ denotes the probability. The uniqueness has been proved.

\section{Proof of the existence for the solution of Eq.\ \eqref{eq:3.1-1}}
Let $Y^{(0)}_t=u(0)$, $Y^{(k)}_t=Y^{(k)}_t(\omega)$, and
\begin{equation*}
Y^{(k+1)}_t=S(t)u(0)+\int^t_0S(t-s)f(Y^{(k)}_t)\mathrm{d}s+\int^t_0S(t-s)\mathrm{d}B_{H,\mu}(s).
\end{equation*}
Assumption \ref{as:2.1} implies
\begin{eqnarray}\label{eq:B1}
\mathrm{E}\left[\left\|Y^{(k+1)}_t-Y^{(k)}_t\right\|^2\right]&=&\mathrm{E}\left[\left\|
\int^t_0S(t-s)f(Y^{(k)}_s)\mathrm{d}s-\int^t_0S(t-s)f(Y^{(k-1)}_s)\mathrm{d}s\right\|^2\right]\nonumber\\
&\le&Ct
\int^t_0\mathrm{E}\left[\left\|Y^{(k)}_s-Y^{(k-1)}_s\right\|^2\right]\mathrm{d}s,
\end{eqnarray}
where $k\ge1$ and $t\le T$. Let
\begin{eqnarray*}
\theta=
\left\{
\begin{array}{ll}
2H, &\ \rho>\frac{1}{2}, \\
\frac{\gamma}{2\alpha}, &\ 0<\rho\le\frac{1}{2}.
\end{array}
\right.
\end{eqnarray*}
Then
\begin{eqnarray*}
\mathrm{E}\left[\left\|Y^{(1)}_t-Y^{(0)}_t\right\|^2\right]&\le&C\mathrm{E}\left[\left\|S(t)u(0)-u(0)\right\|^2\right]+C\mathrm{E}\left[\left\|\int^t_0S(t-s)f\left(u(0)\right)\mathrm{d}s\right\|^2\right]+Ct^{\theta}\\
&\le& Ct^2\left(1+\mathrm{E}\left[\left\|u(0)\right\|^2\right]\right)+Ct^{\theta}\\
&\le& Ct^{\min\{2,\theta\}},
\end{eqnarray*}
where $C$ depends on $H$, $\alpha$, and $\gamma$. So by induction starting from Eq.\ \eqref{eq:B1}, we obtain
\begin{equation}\label{eq:B2}
\mathrm{E}\left[\left\|Y^{(k+1)}_t-Y^{(k)}_t\right\|^2\right]\le \frac{C^{k+1}t^{2k+\min\{2,\theta\}}}{k!},\quad k\ge0, \ t\in[0,T].
\end{equation}
Let $m>n\ge0$. Using Eq.\ \eqref{eq:B2}, we have
\begin{eqnarray*}
\left\|Y^{(m)}_t-Y^{(n)}_t\right\|_{L^2(D,U)}&=&\left\|\sum^{m-1}_{k=n}Y^{(k+1)}_t-Y^{(k)}_t\right\|_{L^2(D,U)}\\
&\le&\sum^{m-1}_{k=n}\left\|Y^{(k+1)}_t-Y^{(k)}_t\right\|_{L^2(D,U)}\\
&\le&\sum^{m-1}_{k=n}\left(\frac{C^{k+1}T^{2k+\min\{2,\theta\}}}{k!}\right)^{\frac{1}{2}}.
\end{eqnarray*}
As $m,n\to\infty$, then $\left\|Y^{(m)}_t-Y^{(n)}_t\right\|_{L^2(D,U)}=0$. The above deduction shows that $\left\{Y^{(k)}_t\right\}^\infty_{k=0}$ is a Cauchy sequence in space $L^2(D,U)$. Define
\begin{equation*}
u(t):=\lim_{n\to\infty}Y^{(n)}_t.
\end{equation*}
Then $u(t)$ satisfies Eq.\ \eqref{eq:3.1-1}.

\section{Description for the simulation of tempered fractional Brownian motion }
 In this paper, the Ckolesky method \cite{26} is used to simulate tfBm. Suppose $0\le t_1\le \dots\le t_{m}\le \dots\le t_{M}=T(m=1,2,\dots, M-1)$ and the sizes of the mesh $\Delta t=t_{m+1}-t_{m}$. Let's consider the following vector
\begin{equation*}
Z=\left(\beta_{H,\mu}(t_1),\beta_{H,\mu}(t_2)-\beta_{H,\mu}(t_1),\beta_{H,\mu}(t_3)-
\beta_{H,\mu}(t_2),\dots,\beta_{H,\mu}(t_{M-1})-\beta_{H,\mu}(t_M)\right).
\end{equation*}
The probability distribution of the vector $Z$ is normal with mean 0 and the covariance matrix $\Sigma$. Let $\Sigma_{i,j}$ be the element of row $i$, column $j$ of matrix $\Sigma$. By using Eq.\ \eqref{eq:2.1}, we have
\begin{eqnarray*}
\Sigma_{i,j}&=&\mathrm{E}\left[\left(\beta_{H,\mu}(t_i)-\beta_{H,\mu}(t_{i-1})\right)
\left(\beta_{H,\mu}(t_j)-\beta_{H,\mu}(t_{j-1})\right)\right]\\
&=&\mathrm{E}\left[\beta_{H,\mu}(t_i)\beta_{H,\mu}(t_j)+\beta_{H,\mu}(t_{i-1})\beta_{H,\mu}(t_{j-1})
-\beta_{H,\mu}(t_i)\beta_{H,\mu}(t_{j-1})-\beta_{H,\mu}(t_j)\beta_{H,\mu}(t_{i-1})\right]\\
&=&\frac{1}{2}\left[C^2_{(i-j+1)\Delta t}|(i-j+1)\Delta t|^{2H}+C^2_{(i-j-1)\Delta t}|(i-j-1)\Delta t|^{2H}-2C^2_{(i-j)\Delta t}|(i-j)\Delta t|^{2H}\right].
\end{eqnarray*}
When $\Sigma$ is a symmetric positive matrix, the covariance matrix $\Sigma$ can be written as $L(M)L(M)'$, where the matrix $L(M)$ is lower triangular matrix and the matrix $L(M)'$ is the transpose of $L(M)$. Let $V=(V_1,V_2,\dots,V_M)$. The elements of the vector $V$ are a sequence of independent and identically distributed standard normal random variables. Because $Z=L(M)V$, then $Z$  can be simulated.
Let $l_{i,j}$ be the element of row $i$, column $j$ of matrix $L(M)$. That is,
\begin{equation*}
\Sigma_{i,j}=\sum^j_{k=1}l_{i,k}l_{j,k}, \quad j\le i.
\end{equation*}
As $i=j=1$, we have $l^2_{1,1}=\Sigma_{1,1}$. The $l_{i,j}$ satisfies
\begin{eqnarray*}
l_{i+1,1}&=&\frac{\Sigma_{i+1,1}}{l_{1,1}},\\
l^2_{i+1,i+1}&=&\Sigma_{i+1,i+1}-\sum^{i}_{k=1}l^2_{i+1,k},\\
l_{i+1,j}&=&\frac{1}{l_{j,j}}\left(\Sigma_{i+1,j}-\sum^{j-1}_{k=1}l_{i+1,k}l_{j,k}\right),\quad1<j\le i.
\end{eqnarray*}

\end{document}